\documentclass[11pt,a4paper,twoside]{article}
\usepackage{amssymb}
\usepackage{amsfonts}
\usepackage{geometry}
\usepackage{color}
\usepackage{mathtools}
\usepackage{amsmath,amsthm,amssymb,amsfonts,enumitem}
\usepackage{CJK}
\usepackage{latexsym}
\usepackage{graphicx}
\usepackage{pgfpages}
\usepackage{mathrsfs}
\usepackage{ytableau}
\usepackage{enumerate}
\usepackage{cite}
\usepackage[hyphens]{url}
\usepackage[bookmarks=false,hidelinks]{hyperref}
\usepackage[affil-it]{authblk}
\usepackage{appendix}
\usepackage{amsmath}
\usepackage{comment}
\newcommand{\sgn}{\mathrm{sgn}}
\newtheorem{definition}{Definition}[section]
\newtheorem{theorem}[definition]{Theorem}
\newtheorem{lemma}[definition]{Lemma}
\newtheorem{proposition}[definition]{Proposition}
\newtheorem{corollary}[definition]{Corollary}

\newtheorem{condition}{Assumption}

\numberwithin{equation}{section}

\begin{document}

\title{Large deviations for invariant measure of the stochastic Allen-Cahn equation with inhomogeneous boundary conditions and multiplicative noise}
\author[a]{Rui Bai}
	\author [a] {Chunrong Feng}
	\author[a,b]{Huaizhong Zhao}
	\affil[a]{Department of Mathematical Sciences, Durham University, DH1 3LE, UK}
	\affil[b] {Research Centre for Mathematics and Interdisciplinary Sciences, Shandong University, Qingdao 266237, China}
	
	\affil[ ]{rui.bai@durham.ac.uk, chunrong.feng@durham.ac.uk,  huaizhong.zhao@durham.ac.uk}
	\date{}
	
\maketitle

\begin{abstract}	
    \medskip
	We establish a small-noise large deviation principle for the family of invariant measures \(\{\mu_\epsilon\}_{\epsilon>0}\) associated with the one-dimensional stochastic Allen–Cahn equation, subject to inhomogeneous Dirichlet boundary conditions and driven by unbounded multiplicative noise. The main novelty is that the deterministic system is only weakly dissipative, while the noise coefficient is allowed to have strictly sublinear growth arbitrarily close to linear. Using L. Simon’s convergence theorem, we prove that every trajectory of the corresponding noiseless equation converges, as time tends to infinity, to the unique minimiser of the Ginzburg–Landau energy functional determined by the boundary conditions. A key ingredient is an exponential estimate for the invariant measures outside bounded subsets of \(W^{k^\star,p^\star}\), where \(k^\star p^\star>1\) and \(p^\star\) are sufficiently large; such subsets are compact in the underlying space of continuous functions. As a consequence of the large deviation principle, we show that, as \(\epsilon\to0\), the invariant measures \(\mu_\epsilon\) concentrate exponentially fast around the unique minimiser.
    \vskip10pt
    
    \noindent
{\bf MSC2020 subject classifications:} Primary 60H15, 60F10, 35R60; secondary 37A50, 37L55.
\vskip2pt


\end{abstract}

\tableofcontents
\section{Introduction}
In this paper, we study the following equation
\begin{equation}\label{ACE1}
     \begin{cases}
         \partial_t u_\epsilon(t,\xi) = \Delta u_\epsilon(t,\xi)  + u_\epsilon(t,\xi) -u_\epsilon(t,\xi)^3 + \epsilon^{\frac{1}{2}} g(\xi, u_\epsilon(t,\xi))\partial_t W(t,\xi),  \xi \in (-L,L),\\
         u_\epsilon(0,\xi)=x(\xi),\\
         u_\epsilon(t,\xi) = \sgn(\xi) , \, |\xi| = L, \; t\geq 0,    
         \end{cases}
 \end{equation}
where $W(t,\xi):= \sum_{k \in \mathbb{N}}e_k(\xi)\beta_k(t)$, $\{e_k\}$ is a complete orthogonal basis in $H:=L^2(-L,L)$, and $\{\beta_k(t)\}$ is a sequence of independent Brownian motions. Assume that the mapping $g: [-L,L]\times\mathbb{R} \rightarrow  \mathbb{R}$, is measurable, and $g(\cdot,\cdot): [-L,L]\times\mathbb{R} \rightarrow \mathcal{L} (\mathbb{R})$, is continuous.  
Moreover, for any $\theta,\vartheta \in \mathbb{R}$ $$\sup_{\xi \in [-L,L]}|g(\xi,\theta) -g(\xi,\vartheta)\vert \leq \Psi|\theta-\vartheta\vert ,$$ for some $\Psi >0$. In particular, $g(\cdot,\xi)$ is at most linear growth.

To give a precise meaning to the solution of the above equation, we set 
\begin{equation}
    \psi(\xi)= \xi/L,\, \xi\in [-L,L]
    \end{equation}
 and denote by $S(t)$ the heat semigroup on $(-L,L)$ with zero boundary conditions at the endpoints. We denote by $\mathcal{E}$ the set of all $u \in C[-L,L]$ with $u(-L)=-1,\,u(L)=1$ and by $E$ the space of all $u \in C[-L,L]$ with $u(-L)=u(L)=0$. Note that $\mathcal{E}$ is not a linear space. Then, the mild solution of Eq.(\ref{ACE1}) is defined as $u_\epsilon(t):= \bar{u}_\epsilon(t)+\psi$ and $\bar{u}_\epsilon $ satisfies
\begin{align*}
    \bar{u}_\epsilon(t) &= S(t)(x-\psi) + \int_{0}^{t}S(t-s)\left[(\bar{u}_\epsilon(s)+\psi)-(\bar{u}_\epsilon(s)+\psi)^3 \right]\, ds \\&\ \ \ \ +\epsilon^{\frac{1}{2}}\int_{0}^{t}S(t-s)G(\bar{u}_\epsilon(s))\,d W(s),
    \end{align*}
is the mild solution of the following equation 
\begin{equation}\label{ACE2}
    \begin{cases}
        \partial_t \bar{u}_\epsilon(t,\xi) = \Delta \bar{u}_\epsilon(t,\xi)  + F(\bar{u}_\epsilon(t))(\xi)+ \epsilon^{\frac{1}{2}} G( \bar{u}_\epsilon(t))\partial_t W(\xi),  \xi \in (-L,L),\\
         \bar{u}_\epsilon(0)=x - \psi,\\
         \bar{u}_\epsilon(t,\xi) = 0 , \, |\xi| = L, \; t\geq 0.   
         \end{cases}
\end{equation}
Here, $G(z)\phi(\xi):=g(\xi,z(\xi)+\psi(\xi))\phi(\xi)$ and $F(z)(\xi):= -(z(\xi)+\psi(\xi))^3 + (z(\xi)+\psi(\xi))$, $z \in E$. Note that for each $z$ that satisfies the zero boundary condition, $F(z)$ also satisfies the zero boundary condition, so the mild solution of Eq.(\ref{ACE2}) is well defined. 

According to \cite{cerrai2003stochastic} (see Theorem 5.3, Theorem 5.5, and Proposition 6.1), for any initial datum $x \in E$ and $T>0$ , there exists a unique mild solution $\bar{u}^x_\epsilon\in L^p(\Omega;C([0,T];E) )$ with $p>1$. Moreover, there exists $c_p>0$ such that
$$\mathbb{E}\sup_{t\geq 0}\|\bar{u}^x_\epsilon(t)\Vert^p_{E} \leq c_p(1+ \|x - \psi\Vert^p_{E}).$$
Furthermore, due to the Krylov–Bogoliubov theorem, there exists a sequence $\{t_n\} \uparrow +\infty$ such that the sequence of probability measures 
$$\bar{\mu}_{\epsilon,n}(B):= \frac{1}{t_n}\int_{0}^{t_n}\mathbb{P}(\bar{u}_\epsilon^0(s) \in B)\, d s, \,B \in \mathcal{B}(E)$$ converges weakly to some measure $\bar{\mu}_\epsilon$, which is an invariant measure of Eq.(\ref{ACE2}) (see \cite{cerrai2003stochastic}, Theorem 6.2). We assume the following condition:
\begin{condition}\label{assumption 3}
    Assume that 
    \begin{equation}\label{inf g}
        \inf_{(\xi,\theta) \in [-L,L] \times\mathbb{R}}|g(\xi,\theta)\vert = g_0 >0.
        \end{equation}\end{condition}
By Theorem 3.1 of \cite{maslowski2000probabilistic} and Chapter 7 of \cite{da1996ergodicity} (see also Section 5 of \cite{goldys200512}), under the above assumption, for every $ \epsilon>0$ the invariant measure $\bar{\mu}_\epsilon$ is unique. We define an invariant measure $\mu_\epsilon$ for Eq.(\ref{ACE1}) by 
\begin{equation}\label{mu_epsilon}
    \mu_\epsilon(B+\psi):=\bar{\mu}_\epsilon(B),\;B \in \mathcal{B}(E).
\end{equation} 
Here, $B+\psi:=\{x \in \mathcal{E}:x=y+\psi, y \in E\}$.
The support of this measure is contained in $\mathcal{E}$. The purpose of this paper is to study the LDP problem for $\mu_\epsilon$.

The study of small noise LDP for invariant measures of SPDEs began in the work of Sowers \cite{sowers1992large}. Cerrai and Röckner  considered the stochastic reaction-diffusion equation with a strongly dissipative nonlinear term in a multi-dimensional space (\cite{cerrai2005large}). Brzeźniak and Cerrai considered the stochastic Navier-Stokes equation on $\mathbb{T}^2$ with additive noise (\cite{brzezniak2017large}). Wang considered the stochastic reaction-diffusion equation on an unbounded domain under strongly dissipative conditions (\cite{wang2024large}). In all these cases, the systems have $0$ as their global attractor in the underlying space, following a dissipative argument. In our case, the dynamics are only weakly dissipative, and the existing arguments fail to work for our situation.  The specific potential $V(u)=\frac14(u^2-1)^2$ we consider in this paper has double wells with 2 critical points $\pm 1$ where the potential attains its minimum. Due to the boundary condition of our problem, there exists a unique solution for the following equation 
\begin{equation}\label{eqn:intro}
\begin{cases}
    \Delta m_L(\xi)= V^\prime(m_L(\xi)),\\
    m_L(\pm L)=\pm 1.
\end{cases}
\end{equation}
Moreover, the solution of the noiseless system $u_0(t)$ (i.e. $\epsilon=0$ in equation (\ref{ACE1})) is the $L^2$ gradient flow of the energy functional 
\begin{equation}\label{Energy}
 \mathbf{E}_L(u):= \int_{-L}^L \left[\frac{|u^\prime(\xi)|^2}{2} + V(u(\xi))\right] \; d\xi,\; u(\pm L)=\pm 1.
 \end{equation}
 By using L. Simon's convergence theorem, we show that for any $x \in E$, $z^x_0(t)$ converges to the minimiser of $\mathbf{E}_L $, which is the unique solution of (\ref{eqn:intro}), as time goes to infinity. The minimiser $m_L$ is the stationary profile with two pure phases of alloys, magnetisation, or fluids coexisting at the two ends of a bounded interval. In these contexts, the boundary models the effects of two different materials, magnetisation, or fluids applied to the endpoints.

Note that in the case where the noise is space-time white noise, the invariant measure $\mu$ has an exponential density with respect to the Brownian bridge measure \cite{reznikoff2005invariant}. By the exponential tilting argument, as an alternative version of the Varadhan Lemma, the LDP for the invariant measure will follow immediately (see Theorem III.17 of \cite{hollander2000large}). In the case where $L(\epsilon)$ depends on $\epsilon$ and goes to infinity as $\epsilon \rightarrow0$, the limiting behaviour of the invariant measure $\mu_\epsilon $ depends on the rate at which $L(\epsilon)$ diverges. Such problems are considered by Weber \cite{weber2010sharp}, Otto-Weber-Westdickenburg \cite{otto2014invariant}, Bertini-Stella-Butt{\`a} \cite{bertini2008dobrushin}, and Bertini-Butt{\`a}-Ges{\`u} \cite{bertini2025asymptotics}. For all the works mentioned above, the noise must be additive space-time white noise.
In the case of the multiplicative noise we consider here, the density and tilted LDP arguments break down. To the best of our knowledge, this is the first invariant-measure LDP for the stochastic Allen–Cahn equation with inhomogeneous boundary conditions and unbounded multiplicative noise in a setting where the deterministic dynamics are not strongly dissipative.

A crucial step in proving a large deviation principle for invariant
measures of SPDEs is to establish an exponential estimate for the
invariant measure \(\mu_\epsilon\) outside compact subsets of the state space $E=C([-L,L])$.
In the present work, the equation is driven by multiplicative
space-time white noise whose intensity \(g\) is non-degenerate,
Lipschitz continuous, and satisfies
\begin{equation}\label{eqn:assumption rho}
   |g(\xi,\sigma)|
\leq
C\bigl(1+|\sigma|^\rho\bigr),
\qquad 0\leq\rho<1. 
\end{equation}
Thus, the noise coefficient may be unbounded and may have growth arbitrarily close to linear. It is worth emphasising that the Cerrai–Röckner approach to invariant-measure large deviations for reaction–diffusion equations with cubic dissipative nonlinearities requires the more restrictive condition \(\rho<1/3\). Our method therefore extends the admissible growth range from \(\rho<1/3\) to the full strictly sublinear regime \(\rho<1\). 

Obtaining the required
exponential estimate is particularly delicate in our setting because
the deterministic system is not strongly dissipative, and the driving
space–time white noise has low spatial regularity.
In \cite{cerrai2022large}, bounded subsets of \(H^1\) are used as
compact sets in the exponential-tightness argument. Wang
\cite{wang2024large}, considering trace-class additive noise on
\(H^1(\mathbb R)\), uses intersections of bounded subsets of \(H^1\)
and a weighted \(L^2\)-space. In these works, the required exponential
estimates are obtained by applying It\^{o}'s formula to suitable energy
functionals. This approach is not available in the present setting
because of the lack of strong dissipativity, the low regularity of the
noise, and the unbounded multiplicative coefficient.

The proof of exponential tightness consists of two steps. First, we establish an exponential estimate for \(\mu_\epsilon\) outside bounded subsets of \(E\). Second, we derive a finite-time exponential estimate outside bounded subsets of \(W^{k^\star,p^\star}(-L,L)\), which are compact in \(E\) whenever \(k^\star p^\star>1\). Combining the two estimates with the invariance of \(\mu_\epsilon\) yields the required exponential estimate outside compact subsets of \(E\).
For the finite-time compact estimate, motivated by the semigroup argument used in our previous work \cite{BAI2026111284}, we exploit the smoothing property of the heat semigroup to reduce the \(W^{k^\star,p^\star}\)-estimate of the solution to an estimate for the stochastic convolution. The main new difficulty is the unbounded multiplicative coefficient. To address this, we introduce a stopping-time localisation: before the solution exits a bounded subset of \(E\), the stopped noise coefficient is uniformly bounded. The factorisation method and the Boué–Dupuis variational representation then yield the required estimate for the stopped stochastic convolution, while the corresponding exit probability is controlled separately by a variational argument. This yields the finite-time exponential estimate outside
bounded subsets of \(W^{k^\star,p^\star}(-L,L)\).

The estimate for the invariant measure outside bounded subsets of \(E\)
is obtained by a different argument. We first use the variational
representation to derive a finite-time exponential estimate in the
\(E\)-norm, which is valid for every \(0<\epsilon\leq1\), with the
admissible range of \(\epsilon\) independent of the exponential level
\(r\). This uniformity is essential since an estimate obtained only
from an asymptotic LDP would generally be valid only for
\(0<\epsilon\leq\epsilon_r\), where \(\epsilon_r\) may depend on \(r\).
We then combine this non-asymptotic estimate with a comparison
inequality for the solution and iterate it over successive time
intervals and spatial annuli. This yields an exponential estimate
uniformly for all \(t>0\), and hence, by invariance, an exponential
estimate for \(\mu_\epsilon\) outside bounded subsets of \(E\).

Finally, we start the process with the initial distribution \(\mu_\epsilon\)
and split the probability according to whether the initial condition
belongs to a sufficiently large bounded subset of \(E\). The
contribution from its complement is controlled by the invariant-measure
estimate in \(E\), while the contribution from the bounded part is
controlled by the finite-time \(W^{k^\star,p^\star}\)-estimate. Since
bounded subsets of \(W^{k^\star,p^\star}(-L,L)\) are compact in \(E\),
this proves the exponential tightness of the family
\(\{\mu_\epsilon\}\). In particular, to the best of our knowledge, this provides the first small-noise exponential estimate for invariant measures of reaction–diffusion equations driven by unbounded multiplicative space–time white noise.
The resulting argument combines variational estimates, a stopping-time localisation, the factorisation method, a comparison inequality, and an annular iteration (\ref{eqn:iteration 2}). It does not rely on the construction of a Lyapunov function or on the application of Itô’s formula to an energy functional. The variational method developed here is not specific to the model considered in this paper. It has also been applied, in work in preparation, to the invariant-measure LDP for the stochastic Burgers equation studied in our previous paper \cite{BAI2026111284}, where it yields a substantial improvement over the earlier result. More broadly, the present framework applies to additive noise, bounded multiplicative noise, and unbounded multiplicative coefficients satisfying (\ref{eqn:assumption rho}), and may therefore be useful for studying invariant-measure large deviations for other weakly dissipative SPDEs.
It therefore extends the admissible growth range beyond the condition \(\rho<1/3\) obtained in the earlier Cerrai–Röckner approach. We emphasise, however, that the borderline case \(\rho=1\) is not covered by the present argument and remains open.

\section{LDP for dynamics}\label{sec:LDP Dynamics}
The large deviation result for the dynamics of Eq.(\ref{ACE2}) has been well studied by Cerrai and R{\"o}ckner in \cite{cerrai2004large}. We formulate their results (see \cite{cerrai2004large}, Theorems 6.2 and 6.3) to obtain the Freidlin-Wentzell uniform large deviations principle (FWULDP) for the dynamics of Eq.(\ref{ACE2}):\\
1. For any  $ z \in C([0,T];E)$ and $\delta, \gamma >0$, there exists $\epsilon_0 >0$ such that for any $\epsilon< \epsilon_0$,
\begin{equation}\label{LDP-Dynamics-Lower}
    \mathbb{P}(\|\bar{u}_\epsilon^x-z\Vert_{C([0,T];E)} < \delta)\geq \exp\left( -\frac{I^x_T(z)+\gamma}{\epsilon}\right),
\end{equation} uniformly with respect to $\|x\Vert_{E} \leq R$.\\
2. For any $ r\geq 0$ and $\delta, \gamma >0$, there exists $\epsilon_0 >0$ such that for any $\epsilon< \epsilon_0$,
\begin{equation}\label{LDP-Dynamics-Upper}
    \mathbb{P}(d_{C([0,T];E)}(\bar{u}_\epsilon^x,K_T^x(r)) \geq \delta)\leq \exp\left( -\frac{r-\gamma}{\epsilon}\right),
\end{equation} uniformly with respect to $\|x\Vert_{E }\leq R$, where $K_T^x(r):= \{z \in C([0,T];E): I_T^x(z) \leq r\}$.

For every $z \in C({0,T};E)$, the rate functional $I_T^x$ is given by 
$$I_T^x(z):=\frac{1}{2}\inf\left\{\int_0^T\|f(t)\Vert_{H}^2\, dt; f \in L^2(0,T;H), z= z^x_f\right\},$$ with the convention that $\inf\varnothing = +\infty$.
Here, $z^x_f$ is the solution of the following skeleton equation,
%\begin{equation}\label{Skeleton equation-1}
%\begin{cases}
    %\partial_tz(t,\xi)=\Delta z(t,\xi) + z(t,\xi) - z(t,\xi)^3 + g(t,\xi,z(t,\xi))f(t),\; \xi\in (-L,L),\\
    %z(0)=x,\\
    %z(t,\xi) = \sgn(\xi), |\xi\vert =L,\; t \geq 0.
    %\end{cases}
%\end{equation}
%The solution of Eq.(\ref{Skeleton equation-1}) is in the sense of $z(t):= \bar{z}(t) + \psi$, where $\bar{z}$ is the solution of equation,
\begin{equation}\label{Skeleton equation-2}
    \begin{cases}
        \partial_tz(t,\xi) = \Delta z(t,\xi) + F(z(t))(\xi) + G(z(t,\xi))f(t)(\xi), \xi \in (-L,L),\\
        z(0,\xi)=x(\xi) ,\\
        z(t,\xi)=0,\;|\xi\vert = L, \;t \geq 0,
    \end{cases}
\end{equation}
for a given function $f \in L^2(0,T;H)$.
Theorem 4.1 in \cite{cerrai2004large} implies that for any $f \in L^2(0,T;H)$ and $x \in E$, there exists a unique mild solution to Eq.(\ref{Skeleton equation-2}) in $C([0,T];E)$. Thus, $I_T^x$ is well defined. Moreover, according to Theorem 5.1 of \cite{cerrai2004large}, the level sets $K_T^x(r)$ are compact. For any fixed $\phi \in L^2(0,T;H)$ and $z \in C([0,T];E)$ we define 
\begin{equation}\label{Z2}
    Z_\phi(z)(t) := \int_{0}^t S(t-s)G(z(s))\phi(s)\,ds, \; t \in [0,T].
\end{equation}
According to Section 4 of \cite{cerrai2004large},
\begin{equation}\label{Y-norm}
    \|Z_\phi(z)\Vert_{C([0,T];E)} \leq c(T)\left(1+\|z\Vert_{C([0,T];E)}\right)\|\phi\Vert_{L^2(0,T;H)},
\end{equation}
and
\begin{equation}\label{Y-norm diff 1}
    \|Z_\phi(z_1) - Z_\phi(z_2)\Vert_{C([0,T];E)} \leq c(T)\left(\|z_1-z_2\Vert_{C([0,T];E)}\right)\|\phi\Vert_{L^2(0,T;H)},
\end{equation}
for some continuous increasing function $c(t)$ with $c(0)=0$. Then, there exists $t_0 >0$ such that $c(t_0)\|\phi\Vert_{L^2(0,T;H)}\leq\frac{1}{2} $. We get 
\begin{equation}\label{Y-norm diff 2}
 \|Z_\phi(z_1) - Z_\phi(z_2)\Vert_{C([0,t_0];E)} \leq \frac{1}{2}\|z_1-z_2\Vert_{C([0,t_0];E)}.
 \end{equation}

Before proving the LDP for $\mu_\epsilon$, we also need to show the Dembo-Zeitouni uniform large deviations principle (DZULDP). 
%The following definition can be found in \cite{dembo2009large}.
%\begin{definition}
	%\label{DZULDP}
	%Let $\mathcal{M}$ be a Banach space, and $D$ be some non-empty set. Suppose that for each $x \in D$, $\{\mu^x _\epsilon \}_{\epsilon > 0}$ is a family of probability measures on $\mathcal{M}$ and the action function $I^x:\mathcal{M}\rightarrow [0,+\infty]$ is a good rate function. The family $\{\mu^x_\epsilon \}_{\epsilon > 0}$ is said to satisfy a Dembo-Zeitouni uniform large deviation principle (DZULDP) in $\mathcal{M}$, with rate function $I^x$, uniformly with respect to $x \in D$, if the following estimates hold:
	%\item[(1)](Lower bound) For any $\gamma > 0$ and open set $G \subset \mathcal{M}$, there exists $\epsilon _0 > 0$ such that
	%\begin{equation*}
		%\inf_{x \in D}\mu^x _\epsilon (G)  \geq \exp\left(-\frac{1}{\epsilon }\left[\sup_{y \in D}\inf_{u \in G}I^y(u) + \gamma \right]\right),
	%\end{equation*}
	%for any $\epsilon \leq \epsilon _0$.
	%\item[(2)] (Upper bound) For any $\gamma > 0$ and closed set $F \subset \mathcal{M}$, there exists $\epsilon _0 > 0$ such that 
	%\begin{equation*}
		%\sup_{x \in D}\mu^x _\epsilon (F) \leq \exp \left(-\frac{1 }{\epsilon }\left[\inf_{y \in D } \inf_{u \in F} I^y(u) -\gamma \right]\right),
	%\end{equation*}
	%for any $\epsilon \leq  \epsilon _0$.
% \end{definition}

\begin{corollary}
	\label{corollary DZ2}
	Let $D \subset E$ be a compact set. If $\bar{u}_\epsilon^x$ is the mild solution of Eq.(\ref{ACE2}), then the family $\left\{\mathcal{L} (\bar{u}_\epsilon^x)\right\}_{\epsilon > 0}$ satisfies a Dembo-Zeitouni uniform large deviations principle in $C([0,T];E)$ with a good rate function $I_T^x$, uniformly with respect to $x \in D$. 
 \end{corollary}
\begin{proof}
Since we have FWULDP, according to Theorem 2.7 of \cite{salins2019equivalences}, it suffices to show that for any $x_n \rightarrow x \in E$,

	$$\lim_{n \rightarrow \infty} \max\left(\sup_{z \in K_T ^{x_n}(r)} d_{C([0,T];E)}(z,K_T ^x(r)), \sup_{z \in K_T ^x(r)} d_{C([0,T];E)}(z,K_T ^{x_n}(r))\right) = 0.$$ 
This is implied by the continuity of the skeleton function with respect to the initial conditions. %We define $$\partial\|x\Vert_{E}:= \left\{h \in E^*; \|h\Vert_{E^* }=1, \langle h,x\rangle_{E,E^* } = \|x\Vert_{E}\right\}.$$  
To see this, first since $F$ is locally Lipschitz continuous on $E$, there exists $C_R>0$ such that
$$\|F(x) -F(y)\Vert_E \leq C_R\|x-y\Vert_E,$$ for any $x, y \in B_R(E)$.
Moreover, according to Theorem 4.1 of \cite{cerrai2004large}, for $x,y \in E$ and $\phi,\varphi \in L^2(0,T;H)$, there exists $M$ such that $$\|z_\phi^x\Vert_{C([0,T];E)} +\|z_\varphi^y\Vert_{C([0,T];E)}  \leq c(T,\|\phi\Vert_{L^2(0,T;H)}+\|\varphi\Vert_{L^2(0,T;H)} )(1+\|x\Vert_E+\|y\Vert_E)< M.$$ We set $\rho=z^x_\phi -z^y_\varphi.$ Then, for any $t\in [0,t_0]$

%\begin{align*}
    %\frac{d}{dt}^-\|\rho(t)\Vert_{E} &\leq \langle \Delta \rho(t), \delta_{\rho(t)}\rangle_{E,E^*} + \langle F\left(z_\phi^x(t)\right)-F\left(z_\varphi^y(t)\right),\delta_{\rho(t)}\rangle_{E,E^*}\\ %&\ \ \ \ + \langle G(t,z^x_\phi(t))\phi(t)-G(t,z^y_\varphi(t))\phi(t), \delta_{v(t)}\rangle \\ &\ \ \ \ + \langle G(t,z^y_\varphi(t))\phi(t)-G(t,z^y_\varphi(t))\varphi(t), \delta_{v(t)}  \rangle \\
    %& \leq C_R\|\rho(t)\Vert_E + C_R\|Y_\phi(z_\phi^x)(t) -Y_\varphi(z_\varphi^y)(t)\Vert_E \\ 
    %&\leq C_R\|\rho(t)\Vert_E + C_R\|Y_\phi(z_\phi^x)(t) -Y_\phi(z_\varphi^y)(t)\Vert_E  + \|Y_{\phi-\varphi}(z_\varphi^y)(t) \Vert_E  \\ 
    %& \leq   C_R\|\rho(t)\Vert_E + c(t)
   % \end{align*}

   \begin{align*}
    \|\rho(t)\Vert_E &\leq \|S(t)(x-y)\Vert_E + \int_0^t\|S(t-s)\left(F(z_\phi^x)(s)-F(z_\varphi^y)(s) \right)\Vert_E \, ds \\ &\ \ \ \ + \|Z_\phi(z_\phi^x) -Z_{\varphi}(z_\varphi^y)\Vert_E\\ 
    &\leq \|x-y\Vert_E + \int_0^tC_R\|\rho(s)\Vert_E\, ds  + \|Z_{\phi-\varphi}(z_\varphi^y)(t) \Vert_E \\
    &\ \ \ \ + \|Z_\phi(z_\phi^x)(t) -Z_\phi(z_\varphi^y)(t)\Vert_E \\
    &\leq \|x-y\Vert_E + \int_0^tC_R\|\rho\Vert_{C([0,s];E)}\, ds + c(T)(R+1)\|\phi-\varphi\Vert_{L^2(0,T;H)}\\
    &\ \ \ \ + \frac{1}{2}\|\rho\Vert_{C([0,t_0];E)}.
    \end{align*}
Here, the last inequality is due to (\ref{Y-norm}) and (\ref{Y-norm diff 2}). Thus, 
\begin{align*}
    \|\rho\Vert_{C([0,t_0];E)} &\leq 2\|x-y\Vert_E + 2\int_0^{t_0}C_R\|\rho\Vert_{C([0,t];E)}\, dt \\&\ \ \ \ + 2c(t_0)(R+1)\|\phi-\varphi\Vert_{L^2(0,t_0;H)}.    
\end{align*}
By the Gr\"onwall inequality, we obtain
\[
\|z_\phi^x-z_\varphi^y\|_{C([0,t_0];E)}
\leq
c(t_0,R,M)
\left(
\|x-y\|_E
+
\|\phi-\varphi\|_{L^2(0,t_0;H)}
\right),
\]
where
\[
\|x\|_E+\|y\|_E\leq R,
\qquad
\|\phi\|_{L^2(0,T;H)}
+
\|\varphi\|_{L^2(0,T;H)}
\leq M.
\]
Moreover, the a priori estimate for the skeleton equation implies that
\[
\sup_{t\in[0,T]}
\left(
\|z_\phi^x(t)\|_E
+
\|z_\varphi^y(t)\|_E
\right)
\leq C(T,R,M).
\]
Therefore, the preceding argument can be repeated on each interval
\[
[nt_0,(n+1)t_0]\cap[0,T],
\qquad
n=0,\ldots,\left\lceil\frac{T}{t_0}\right\rceil-1,
\]
with constants depending only on \(T\), \(R\), and \(M\). By iteration, we obtain
\begin{equation}\label{z_phi^x-z_varphi^y}
\|z_\phi^x-z_\varphi^y\|_{C([0,T];E)}
\leq
c(T,R,M)
\left(
\|x-y\|_E
+
\|\phi-\varphi\|_{L^2(0,T;H)}
\right).
\end{equation}
Therefore, for any $z \in K_T^x$, there exists $\phi$ such that $z=z_\phi^x$ and $\frac12\|\phi\Vert_{L^2(0,T;H)}^2\leq r $. Then, we have $$d_{C([0,T];E)}(z,K_T^{x_n}(r)) \leq d_{C([0,T];E)}(z_\phi^x,z_\phi^{x_n})  \leq c(T,\|x\Vert_E,r)\|x-x_n\Vert_E, $$
which implies the result.
\end{proof}

\section{The rate functional of LDP for invariant measures}
\subsection{Energy functional and quasi-potential}
In this section, we study the rate functional $U_L$ for invariant measures. We first introduce the energy functional $\mathbf{E}_L$. For functions $u:[-L,L] \rightarrow \mathbb{R}$ with boundary conditions $u(\pm L)=\pm 1$, we consider the energy functional (\ref{Energy}) with
\begin{equation}
    V(u)= \frac{1}{4}(u^2-1)^2 .
\end{equation}
According to Proposition 3.1 of \cite{bertini2008dobrushin}, the functional $\mathbf{E}_L$ has a unique minimizer $m_L$ in $\mathcal{E}$ and $m_L$ is the unique solution for the equation
\begin{equation}\label{elliptic ACE0}
\begin{cases}
    \Delta m(\xi) = m^3(\xi)-m(\xi) , \xi \in (-L,L)\\
    m(\xi)=\sgn(\xi), |\xi\vert=L.
    \end{cases}
\end{equation}
Equation (\ref{elliptic ACE0}) yields that, for $\xi \in (-L,L)$
\begin{equation}\label{eqn:m_L}
     m^\prime(\xi)^2 = \frac{1}{2}(m^2(\xi)-1)^2+e_L,
     \end{equation} where $e_L>0$ depends on the spatial scale $L$.
Taking (\ref{eqn:m_L}) into account in manipulating (\ref{Energy}), for any $u \in \mathcal{E}$, we have
\begin{equation}\label{E_L align}
\begin{split}
     \mathbf{E}_L(u)&= \int_{-L}^L \frac{|u^\prime(\xi)|^2}{2} + \frac{1}{4}(u^2(\xi)-1)^2 \; d\xi \\ 
    &= \int_{-L}^L \Bigg[ \frac{1}{2}  \left(u^\prime(\xi) - \sqrt{\frac{1}{2}(u^2(\xi)-1)^2+e_L} \right)^2  \\ 
    &\ \ \ \ + \sqrt{\frac{1}{2}(u^2(\xi)-1)^2+e_L} \cdot u^\prime(\xi) -\frac{1}{2}e_L\Bigg]\;d\xi \\ 
    &= \frac{1}{2}  \int_{-L}^L \left(u^\prime(\xi) - \sqrt{\frac{1}{2}(u^2(\xi)-1)^2+e_L} \right)^2 \;d\xi  \\ &\ \ \ \ + \int_{-1}^1 \sqrt{\frac{1}{2}(u^2-1)^2+e_L}\;du - Le_L. \end{split}
     \end{equation}
Then, 
\begin{equation}\label{eqn:Energy m_L}
    \mathbf{E}_L(m_L)= \int_{-1}^1 \sqrt{\frac{1}{2}(u^2-1)^2+e_L}\;du - Le_L
    \end{equation} and for any $u \in \mathcal{E}$,
    $$\mathbf{E}_L(u) \geq \mathbf{E}_L(m_L). $$
We denote for any $\bar{u} \in E$
\begin{equation}\label{Energy 0}
     \mathbf{E}_L^*(\bar{u})= \mathbf{E}_L(\bar{u}+ \psi)-  \mathbf{E}_L(m_L)
\end{equation}
so that $\mathbf{E}_L^*(m_L- \psi) = 0$. Note that $m_L-\psi$ is the only solution for the equation 
\begin{equation}\label{elliptic ACE}
\begin{cases}
    \Delta z(\xi) + F(z)(\xi) = 0,\\
    z(-L)=z(L) = 0.
    \end{cases}
\end{equation}
One can prove the following proposition which will be needed in this paper:
\begin{proposition}
     \label{prop: small energy}
    There exists $\eta >0$ such that if $\bar{u }\in E$ and $\mathbf{E}_L^*(\bar{u}) \leq \eta $, then 
    $$\|\bar{u}-(m_L-\psi)\Vert_E \leq C\sqrt{\eta}.$$
    \end{proposition}
    \begin{proof}
        For $\bar{u} \in E$, we denote $u= \bar{u}+\psi$. According to (\ref{eqn:m_L}), (\ref{E_L align}), (\ref{eqn:Energy m_L}), and (\ref{Energy 0}), $$\mathbf{E}_L^*(\bar{u})= \frac{1}{2}  \int_{-L}^L \left(u^\prime(\xi) - \sqrt{\frac{1}{2}(u^2(\xi)-1)^2+e_L} \right)^2 \;d\xi.$$
       Let us write 
    $$u^\prime(\xi)= V(u(\xi)) + \phi(\xi), u(-L)=-1,$$ where $V(\cdot)=\sqrt{\frac{1}{2}((\cdot)^2-1)^2+e_L} $ is locally Lipschitz. Assume $\mathbf{E}_L^*(\bar{u}) \leq \eta $, then $\phi$ satisfies $\int_{-L}^L\phi^2(\xi)\;d\xi \leq 2\eta $. By the Cauchy-Schwarz inequality,
    $$\int_{-L}^L|\phi(\xi)\vert\;d\xi \leq 2\sqrt{L\eta}. $$ 
     Moreover, from (\ref{E_L align})
        \begin{align*}
            \eta \geq \mathbf{E}_L^*(\bar{u}) &\geq \int_{\inf_{\xi \in [-L,L]}u(\xi)} ^{\sup_{\xi \in [-L,L]}u(\xi)}  V(u)\;du -Le_L - \left(\int_{-1}^1V(u)\;du -Le_L   \right)\\ &\geq \int_{\inf_{\xi \in [-L,L]}u(\xi)} ^{-1}  V(u)\;du  + \int_1 ^{\sup_{\xi \in [-L,L]}u(\xi)}  V(u)\;du          \end{align*}
            This implies that we can choose $\eta$ sufficiently small such that $\sup_{\xi \in [-L,L]}|u(\xi)\vert \leq 2$.
    Thus, by (\ref{eqn:m_L}) and the locally Lipchitz property of $V(\cdot)$, one obtains 
    \begin{align*}
    \frac{d}{d \xi}|u(\xi)-m_L(\xi)\vert &\leq |V(u(\xi)) - V(m_L(\xi))\vert + |\phi(\xi)\vert \\ &\leq C|u(\xi)-m_L(\xi)\vert + |\phi(\xi)\vert.
\end{align*}
    By Gronwall's inequality, one can choose $\eta$ small enough such that $$\|\bar{u}-(m_L-\psi)\Vert_E= \sup_{\xi \in[-L,L]}|u(\xi)-m_L(\xi)\vert \leq 2\sqrt{L\eta}e^{2LC} . $$  This completes the proof.
    \end{proof}
Now, we can define the rate functional. The quasi-potential $U_L:E \rightarrow \mathbb{R}$ for LDP of $\bar{\mu}_\epsilon$ is defined as
\begin{equation}\label{Rate functional}
    U_L(\zeta):= \inf \left\{I_T^{m_L-\psi}(z): T >0, z \in C([0,T];E), z(0)= m_L-\psi, z(T)= \zeta \right\}.
\end{equation}
Note that the level sets of $\mathbf{E}_L^*$ are compact (see Proposition 3.1 of \cite{bertini2008dobrushin}). By the Sobolev embedding theorem, $\mathbf{E}_L^*(\bar{u})$ is finite if and only if $\bar{u} \in H^1(-L,L)$. For $\bar{u}, h \in H^1(-L,L)$, we denote by $D_h\mathbf{E}_L^*(\bar{u})$ the Fr{\'e}chet derivative of $\mathbf{E}_L^*$ on $\bar{u}$, then
\begin{align*}
    D_h\mathbf{E}_L^*(\bar{u}) &= \lim_{t \rightarrow 0}\frac{\mathbf{E}_L^*(\bar{u}+th)- \mathbf{E}_L^*(\bar{u}) }{t} \\ 
    &= \int_{-L}^L [-(\bar{u}+\psi)(\xi)+(\bar{u}+\psi)^3(\xi)]h(\xi)\;d\xi\\
    &\ \ \ \ + \int_{-L}^L (\bar{u}+\psi)^\prime(\xi)h^\prime(\xi) \;d\xi \\
    &=- \int_{-L}^L [\Delta \bar{u}(\xi)+ (\bar{u}+\psi)(\xi)-(\bar{u}+\psi)^3(\xi)]h(\xi) \;d\xi
\end{align*}
Here, the last equality is due to integration by parts and the fact that $\Delta \psi =0$. This implies that $$D_h\mathbf{E}_L^*(u)= \langle D\mathbf{E}_L^*(u), h \rangle_{H^{-1},H^1}, $$ where
\begin{equation}\label{D Energy 0}
    D\mathbf{E}_L^*(u)= -[\Delta u+ (u+\psi)-(u+\psi)^3].
\end{equation}
The functions $U_L$ and $\mathbf{E}_L^*$ have the following connection
\begin{proposition}\label{prop:energy upper bound}
    For any $\zeta \in E$, there exists $C>0$ such that
    \begin{equation} \label{Energy upper bound}
        \mathbf{E}_L^*(\zeta) \leq C U_L(\zeta)(U_L(\zeta)+1).
    \end{equation}
\end{proposition}
\begin{proof}
Suppose $U_L(\zeta)=+ \infty $, then our claim is clearly true. Therefore, we assume $U_L(\zeta) < \infty$. Then, there exist some $T>0$ and $z \in C([0,T];E)$ satisfying $z(0)=m_L-\psi$, $z(T)=\zeta$ such that $z$ is a solution of equation ($\ref{Skeleton equation-2}$) with $f \in L^2(0,T;H)$ such that $$I_T^{m_L-\psi}(z)=\frac{1}{2}\|f\Vert_ {L^2(0,T;H)}^2 < +\infty.$$ Multiplying both parts of equation (\ref{Skeleton equation-2}) by $D \mathbf{E}^*_L(z(t))$ and integrating with respect to $\xi $ in $[-L,L]$, by the chain rule, one has
\begin{align*}
    \frac{d\;\mathbf{E}^*_L(z(t))}{dt}&=  - \int_{-L}^{L}|\Delta z(t,\xi)+F(z(t))(\xi) \vert^2\;d\xi \\
    &\ \ \ \ -  \int_{-L}^L G(t,z(t))f(t,\xi)\cdot[\Delta z(t,\xi)+F(z(t))(\xi)] \; d\xi \\ 
    &\leq -\int_{-L}^L|\Delta z(t,\xi) + F(z(t))(\xi)\vert^2\;d\xi + \int_{-L}^L|\Delta z(t,\xi) + F(z(t))(\xi)\vert^2\;d\xi\\
    &\ \ \ \ +\frac{1}{4}C(\|z(t)\Vert_E+1)^2\|f(t)\Vert_H^2.
    \end{align*}
    Integrating by $t$, one gets
    \begin{equation}\label{Energy inequality}
        \mathbf{E}_L^*(\zeta) - \mathbf{E}^*_L(z(0)) \leq \frac{C}{2}(\|z\Vert_{C([0,T];E)}^2+1)I_T(z).
    \end{equation}
    Following the same procedure of the proof in Proposition 6.1 of \cite{cerrai2003stochastic}, we have \begin{equation}\label{ineqn: skeleton}
        \sup_{t \geq 0}\|z(t)\Vert_E \leq  C( \|f\Vert_{L^2([0,+\infty);H)}+1).
        \end{equation}
   Since $\mathbf{E}_L^*(m_L-\psi) = 0$, taking the infimum over $T$ and $z \in C([0,T];E)$, 
    $$\mathbf{E}_L^*(\zeta) \leq CU_L(\zeta)(U_L(\zeta)+1).$$
\end{proof}

The above proposition shows that the energy functional is controlled by the rate functional. A special case for the skeleton equation is $f = 0$. The following lemma states the global dynamics of $z_0^x$.

\begin{lemma} \label{lemma-energy decreasing}
    For any $x \in H$, let $z_0^x$ be the solution of the skeleton equation (\ref{Skeleton equation-2}) with $f=0$. Then, 
    \begin{equation}\label{equilibrium}
        \lim_{t \rightarrow +\infty}\|z_0^x(t)-(m_L-\psi)\Vert_{H^1(-L,L)} = 0,
    \end{equation}
    and 
    \begin{equation}\label{limit energy}
        \lim_{t \rightarrow +\infty}\mathbf{E}_L^*(z_0^x(t)) = 0.
    \end{equation}
\end{lemma}
\begin{proof}
Multiplying both sides of equation (\ref{Skeleton equation-2}) by $z_0^x(t)$ and integrating with respect to $\xi$, we have
\begin{align*}
   &\ \ \ \  \frac{1}{2}\partial_t\|z_0^x(t)\Vert_H^2 + \|z_0^x(t)\Vert_{H^1(-L,L)}^2\\ &= \int_{-L}^L-(z_0^x(t,\xi)+\psi(\xi))^3\cdot z_0^x(t,\xi) +(z_0^x(t,\xi)+\psi(\xi))\cdot z_0^x(t,\xi) \;d\xi \\ &\leq -\|z_0^x(t)\Vert_H^2 + CL.
\end{align*}
By Gronwall's inequality,
$$\|z_0^x(t)\Vert_H^2 + \int_0^t\|z_0^x(s)\Vert_{H^1(-L,L)}^2\;ds \leq C(t,L)(\|x\Vert_H^2 +1).$$ This implies that there exists $0<t^\prime\leq1$ such that $z_0^x(t^\prime) \in H^1(-L,L)$.
Multiplying both parts of equation (\ref{Skeleton equation-2}) by $D \mathbf{E}^*_L(z^x_0(t))$ and integrating with respect to $\xi $ in $[-L,L]$. Since $f=0$, by chains rule
$$\frac{d\;\mathbf{E}^*_L(z_0^x(t))}{dt} = -\|\Delta z_0^x(t,\cdot) + F(z_0^x(t))(\cdot)\Vert_H^2. $$ Thus, for any $t >t^\prime$,
\begin{equation}\label{Energy equality}
    \mathbf{E}_L^*(z_0^x(t^\prime)) - \mathbf{E}_L^*(z_0^x(t)) = \int_{t^\prime}^t \|\Delta z_0^x(s,\cdot) + F(z_0^x(s))(\cdot)\Vert_H^2\;ds \geq 0.
    \end{equation}
Since $\mathbf{E}_L^*(\zeta)\geq 0$, for any $\zeta \in E$, then for all $t >0$
\begin{equation}\label{gradient L^2 norm}
    \int_{t^\prime}^t \|\Delta z_0^x(s,\cdot) + F(z_0^x(s))(\cdot)\Vert_H^2\;ds \leq \mathbf{E}_L^*(z_0^x(t^\prime)) < + \infty, 
\end{equation}
and 
$$\mathbf{E}_L^*(z_0^x(t^\prime)) \geq \mathbf{E}^*_L(z^x_0(t))\geq \frac{1}{2}\|z_0^x(t)+\psi\Vert_{H^1}^2 - \mathbf{E}_L(m_L)\geq \|z_0^x(t)\Vert_{H^1}^2 - \frac{1}{L}-\mathbf{E}_L(m_L).$$
Let $\{t_n\}_{n\in\mathbb N}$ be an arbitrary sequence such that
$t_n\to+\infty$. By the uniform $H^1$-bound and the compact embedding
$H^1(-L,L)\hookrightarrow E$, there exists a subsequence, still denoted
by $\{t_n\}$, and some $z_L\in H^1(-L,L)$ such that
\[
    z_0^x(t_n)\rightharpoonup z_L
    \quad\text{weakly in }H^1(-L,L),
\]
and
\[
    z_0^x(t_n)\to z_L
    \quad\text{strongly in }E.
\]
Moreover, the preceding argument shows that
\[
    z_0^x(t_n)\to z_L=m_L-\psi
    \quad\text{strongly in }H^1(-L,L).
\]
Since the sequence $\{t_n\}$ was arbitrary, it follows that
\[
    \{z_0^x(t):t\geq1\}
\]
is relatively compact in $H^1(-L,L)$. By (\ref{gradient L^2 norm}), there exists a subsequence still denoted as $z_0^x(t_k)$ such that 
$$\lim_{k \rightarrow +\infty}\|\Delta z_0^x(t_k) + F(z_0^x(t_k))\Vert_H = 0.$$
Since $F$ is locally Lipschitz, for $\tilde{h} \in H$
$$\lim_{k \rightarrow +\infty }\langle F(z_0^x(t_k)) - F(z_L), \tilde{h}\rangle_H \leq C\lim_{k \rightarrow +\infty}\langle |z_0^x(t_k)-z_L\vert,\tilde{h}\rangle_H = 0. $$
Then, for $h \in H^1(-L,L)$,
\begin{align*} 
    \lim_{k \rightarrow +\infty}\langle \Delta z_0^x(t_k) +F(z_0^x(t_k)), h\rangle_H&= \lim_{k \rightarrow +\infty}-\langle \nabla z_0^x(t_k), \nabla h\rangle_H + \langle F(z_0^x(t_k)),h\rangle_H\\
    &= -\langle \nabla z_L,\nabla h\rangle_H + \langle F(z_L) , h\rangle_H \\
    &= \langle \Delta z_L + F(z_L), h\rangle =0.
    \end{align*}
This implies $$\Delta z_L + F(z_L) = 0.$$
Due to the uniqueness of the solution for equation (\ref{elliptic ACE}), one obtains 
$$z_L = m_L-\psi .$$ Moreover, 
\begin{align*}
    \lim_{k \rightarrow +\infty} \|z_0^x(t_k)-z_L\Vert_{H^1(-L,L)}^2 &= \lim_{k \rightarrow +\infty}-\langle \Delta z_0^x(t_k) - \Delta z_L, z_0^x(t_k)-z_L\rangle_H \\
    &= \lim_{k \rightarrow +\infty}-\langle \Delta z_0^x(t_k) + F(z_0^x(t_k)), z_0^x(t_k) - z_L\rangle_H\\
    &\ \ \ \ + \lim_{k \rightarrow + \infty}-\langle F(z_L) - F(z_0^x(t_k)), z_0^x(t_k)-z_L\rangle_H \\&=0.
    \end{align*}
    Here, the Cauchy-Schwarz inequality and the local Lipschitz continuity of $F$ are used when the above two limits are taken. Since the sequence $\{t_n\}$ was arbitrary, the same argument shows
that every sequence contained in
\[
    \{z_0^x(t):t\geq1\}
\]
admits a subsequence converging strongly in $H^1(-L,L)$. Hence this
set is relatively compact in $H^1(-L,L)$.

    The above arguments show that the assumptions in Theorem 1.1 of \cite{jendoubi1998simple} are satisfied. Applying L.Simon's convergence theorem, we have that $z(t)$ is convergent to one equilibrium of the system which satisfies equation (\ref{elliptic ACE}) as $ t\rightarrow \infty$. By the uniqueness of the limit,
$$ \lim_{t \rightarrow +\infty}\|z_0^x(t)-(m_L-\psi)\Vert_{H^1(-L,L)}=0.$$ 
Thus, by the Sobolev embedding,
$$\lim_{t \rightarrow +\infty}\mathbf{E}_L^*(z_0^x(t)) = \mathbf{E}_L^*(m_L-\psi)=0. $$

\end{proof}

\begin{proposition}\label{prop:energy lower bound}
   Under Assumption \ref{assumption 3}, for any $\zeta \in E$, 
     \begin{equation}\label{Energy lower bound}
         g_0^2\cdot U_L \leq 2\mathbf{E}^*_L(\zeta).
     \end{equation}
    
\end{proposition}
\begin{proof}
    According to Lemma \ref{lemma-energy decreasing}, (\ref{gradient L^2 norm}) and the Sobolev embedding theorem, for any $1 \geq \varepsilon >0$ and $\zeta \in E$, there exists a real number $t^*>0$ such that 
    $$\|z_0^\zeta(t^*) -(m_L-\psi)\Vert_E + \|z_0^\zeta(t^*)-(m_L-\psi)\Vert_H \leq \varepsilon,$$ and 
    $$\|\Delta z_0^\zeta(t^*,\cdot)+ F(z_0^\zeta(t^*))(\cdot)\Vert_H \leq \varepsilon.$$
    Define $$v(s)=(1-s)(m_L-\psi)+s\cdot z_0^\zeta(t^*)= (m_L-\psi) + s(z_0^\zeta(t^*)-(m_L-\psi) ), \; s \in [0,1].$$ By definition, $v$ is a path in $C([0,1];E)$ that satisfies $v(0)= m_L-\psi$ and $v(1)= z_0^\zeta(t^*)$. 

    Note that, under our assumptions, the following more explicit expression for $I_T^x$ can read as:
    \begin{equation*}
       I_T^x(z)= 
          \frac{1}{2} \int_0^T \left\|\frac{\partial_tz - \Delta z(t) - F(z(t))}{G(z(t))}\right\Vert_H^2 \;dt.
    \end{equation*}  
Thus, 
    \begin{align*}
        g_0^2\cdot I_1^{m_L-\psi}(v) &\leq \frac{1}{2}\int_0^1\|\partial_sv-\Delta v(s) -F(v(s))\Vert_H^2\;ds\\
        & \leq  \|z_0^\zeta(t^*)-(m_L-\psi)\Vert_H^2 +  \int_0^1\|\Delta v(s) + F(v(s))\Vert_H^2\;ds \\
        &\leq \varepsilon + 2\int_0^1\| \Delta (m_L-\psi)+ F(m_L-\psi)\|_H^2 + \|s[\Delta z_0^\zeta(t^*) - \Delta(m_L-\psi)]\Vert_H^2 \\&\ \ \ \ + \|F(v(s))-F(m_L-\psi)\Vert_H^2 \;ds \\ &= \varepsilon + 2\int_0^1\left[  \|s(\Delta z_0^\zeta(t^*) - \Delta (m_L-\psi))\Vert_H^2  + \|F(v(s))-F(m_L-\psi)\Vert_H^2 \right]\;ds .    %\frac{1}{2}\int_0^1\|G(s,v(s))f(s)\Vert_H^2 \;ds 
    \end{align*}
Since $F(\cdot)$ is locally Lipschitz continuous, 
$$ \|F(z_0^\zeta(t^*))-F(m_L-\psi)\Vert_H  \leq CL\varepsilon, $$ and 
$$\int_0^1 \|F(v(s))-F(m_L-\psi)\Vert_H^2 \;ds \leq C^2L^2\varepsilon^2.$$ Here, the last inequality is due to $\|v(s)- (m_L-\psi)\Vert_E \leq \varepsilon$, $s \in [0,1]$. Then, 
\begin{align*}
    \|\Delta z_0^\zeta(t^*) -\Delta (m_L-\psi)\Vert_H &\leq \|\Delta z_0^\zeta(t^*) + F(z_0^\zeta(t^*))- [\Delta (m_L-\psi) + F(m_L-\psi)]\Vert_H \\ &\ \ \ \ + \|F(z_0^\zeta(t^*))- F(m_L-\psi)\Vert_H\\
    & \leq \|\Delta z_0^\zeta(t^*) + F(z_0^\zeta(t^*))\Vert_H + CL\varepsilon \leq (CL+1)\varepsilon.
    \end{align*}
According to the above discussions, 
\begin{equation}\label{path energy 1}
    g_0^2\cdot I_1^{m_L-\psi}(v) \leq (CL+1)^2\varepsilon.
\end{equation}
Define $$\tilde{v}(s)= z_0^\zeta(t^*+1 -s),\; s \in [1,t^*+1].$$ We have $\tilde{v}(1)= z_0^\zeta(t^*)$ and $\tilde{v}(t^*+1)= \zeta$. Moreover, 
$$\partial_s\tilde{v} = -\Delta z_0^\zeta(t^*+1-s) - F(z_0^\zeta(t^*+1-s)).$$
Thus, 
\begin{equation} \label{path energy 2}
    g_0^2\cdot I_{[1,t^*+1]}^{z_0^\zeta(t^*)}(\tilde{v}) \leq 2 \int_0^{t^*}\|\Delta z_0^\zeta(s) + F(z_0^\zeta(s))\Vert_H^2 \; ds  \leq 2 \mathbf{E}_L^*(\zeta),
    \end{equation}
    where the last inequality is by (\ref{gradient L^2 norm}). 
Therefore, for any $\zeta \in H^1(-L,L)$, there exists $z \in C([0,t^*+1];E)$ defined as
\begin{equation*}
    z(t)=\begin{cases}
        v(t),\; t \in [0,1],\\
        \tilde{v}(t), \; t \in [1,t^*+1],
    \end{cases}
\end{equation*}
such that $z$ satisfies $z(0)= m_L-\psi$ and $z(t^*+1) = \zeta$. Moreover, by (\ref{path energy 1}) and (\ref{path energy 2}), for any $\zeta \in H^1(-L,L)$,  $\varepsilon>0$, there exists $t^*>0$ such that
$$g_0^2\cdot I_{t^*+1}^{m_l-\psi}(z) \leq g_0^2\cdot I_{1}^{m_L-\psi}(v) + g_0^2\cdot I_{[1,t^*+1]}^{z_0^\zeta(t^*)}(\tilde{v}) \leq 2 \mathbf{E}_L^*(\zeta)+ C(L+1)\varepsilon.$$
%Taking infimum over $t^*$ and $z$, one get 
This is then followed by
$$g_0^2\cdot U_L \leq 2 \mathbf{E}^*_L(\zeta).$$

\end{proof}

\subsection{Compactness of level sets of the quasi-potential }
In this section, we will show the compactness of level sets of quasi-potential. For any $t_1<t_2$ and $z \in C([t_1,t_2];E)$, define 
$$I_{t_1,t_2}(z):=\frac{1}{2}\inf\{\|f\Vert_{L^2(t_1,t_2;H)}^2; f \in L^2(t_1,t_2;H), z= z_f\},$$ where $z_f$ is the solution of the skeleton equation (\ref{Skeleton equation-2}) in $[t_1,t_2]$ with $z_f(t_1) = z(t_1)$. We write $I_{-t}$ when $t_2 =0$ and $t_1 = -t <0$. 

According to the proof of Theorem 5.1 in \cite{cerrai2004large}, for any compact set $\Lambda \in E$ the level set
$$K_{\Lambda,t_1,t_2}(r):= \{z \in C([t_1,t_2];E), z(t_1)\in \Lambda, I_{t_1,t_2}(z) \leq r\}$$ is compact.

Similarly, for $z \in C((-\infty,0];E)$ we define $$I_{-\infty}(z):=\frac{1}{2}\inf\{\|f\Vert _{L^2(-\infty,0;H)}^2;f \in L^2(-\infty,0;H), z=z_f  \}= \sup_{t \geq 0}I_{-t}(z),$$ and for $r \geq 0$
$$K_{-\infty}(r):=\{z \in C((-\infty,0];E),I_{-\infty}(z) \leq r, \lim_{t \rightarrow +\infty}\|z(-t)-(m_L-\psi)\Vert_E=0 \}.$$ The following lemma is needed later in this paper.
\begin{lemma}\label{lemma:z-energy}
    For any $r \geq 0$, if $$\liminf_{t \rightarrow +\infty}\|z(-t)\Vert_E < \infty, I_{-\infty}(z) \leq r$$ then
    $$\sup_{t \geq0}\mathbf{E}^*_L(z(-t)) \leq C_r,$$ where $C_r >0$ is a constant depending on $r$.
\end{lemma}
\begin{proof}
    According to the assumption of the lemma, there exists a sequence $t_n \nearrow +\infty$ and a constant $C>0$ such that $\|z(-t_n)\Vert_E \leq C$, $n \in \mathbb{N}$ and $$f\in L^2(-\infty,0;H),\; \sup_{n \in \mathbb{N}}I_{-t_n}(z) = \frac{1}{2} \| f\Vert_{L^2(-\infty,0;H) }^2\leq r,$$ where $z$ is the solution of the equation (\ref{Skeleton equation-2}) with $z(-t_n,\cdot)=z(-t_n)$ . Due to (\ref{ineqn: skeleton}),
    \begin{equation}\label{ineqn: skeleton 2}
        \sup_{t \in [-t_n,0] ,n \in \mathbb{N}}\|z(t)\Vert_E < \infty.    \end{equation}

Then, by standard estimates,
\begin{align*}
       \frac{1}{2}\frac{d}{dt}\|z(t)\Vert_H^2 &= \langle\Delta z(t), z(t)\rangle_H + \langle F(z)(t), z(t)\rangle+ \langle G(z(t))f(t), z(t) \rangle  \\& \leq -\|z(t) \Vert_{H^1}^2 - \kappa\| z(t)\Vert_H^2 + C\|f(t)\Vert_H^2 +C_\kappa.  
       \end{align*} 
       Note that $\|f \Vert_{L^2([-t_n,0];H)}^2= 2I_{t_n}(z_n) \leq 2r$, by Gronwall's inequality, we have for $t \in [-t_n,0]$
       $$
           \|z(t)\Vert_H^2 + 2 \int_{-t_n}^t \|z(s) \Vert_{H^1}^2  \;ds \leq C(t+t_n)(\|z(-t_n)\Vert_H^2 + 2r +1).  $$ Thus, for any $n \in \mathbb{N}$ there exists $t^\prime \in [-t_n,-t_n+1]$ such that $\|z(t^\prime)\Vert_{H^1}^2 \leq C(\|z(-t_n)\Vert_H^2 + 2r +1) $. %In addition, multiplying both parts of equation (\ref{Skeleton equation-2}) by $\partial_tz$ and integrating with respect to $\xi$,
           %\begin{align*}
              % \|\partial_tz\Vert_H^2 &= \langle \Delta z, \partial_t z\rangle + \langle F(z), \partial_tz\rangle+ \langle G(t,z(t))f(t), \partial_tz \rangle \\ & \leq -\frac12\partial_t\|z\Vert_{H^1}^2 +\|F(z)\Vert_H^2 + \frac14\|\partial_tz\Vert_H^2+ \|G(t,z(t))f(t)\Vert_H^2 +\frac14 \|\partial_tz\Vert_H^2 \\&\leq -\frac12\partial_t\|z\Vert_{H^1}^2 + \frac12 \|\partial_tz\Vert_H^2 + C\|f(t)\Vert_H^2 +\|F(z)\Vert_H^2.       \end{align*} 
   % Then, due to (\ref{ineqn: skeleton 2}), for $s \geq t_n$
    %\begin{equation}\label{ineqn: skeleton H^1}
         %\int_{t_n}^s \|\partial_tz\Vert_H^2+  \|z(s)\Vert_{H^1} ^2 \leq   \end{equation}
         From (\ref{Energy}) and the Sobolev embedding, $\mathbf{E}^*_L(z(t^\prime)) \leq C_r$. Then, by (\ref{Energy inequality}), for any $t\geq t^\prime $ $$\mathbf{E}^*_L(z(t))\leq\mathbf{E}^*_L(z(t^\prime))+2C\cdot r \leq C_r .$$ The proof is finished by taking $t_n$ to $-\infty$.
\end{proof}

\begin{proposition}\label{prop:compactness of K}
    For any $r\geq 0$, the set $K_{-\infty}(r)$ is compact in $C((-\infty,0];E)$.
\end{proposition}
Before we prove this proposition, we need the following lemma.
 \begin{lemma}\label{lemma : energy increment}
            Suppose $ \|z(t_0)-(m_L-\psi)\Vert_E \geq \beta $ and $\mathbf{E}_L^*(z(0)) \leq C_r $ for some $\beta,C_r>0$, then there exists some $t_0>0$, such that            $$I_{t_0} (z) >0.$$
        \end{lemma}
\begin{proof}
    We will show that there exists $t_0$ such that $$r_{t_0}:=\inf \{I_{t_0}(z):z \in C([0,t_0];E), \mathbf{E}^*_L( z(0)) \leq C_{r}, \| z(t_0)-(m_L-\psi)\Vert_E \geq \beta \} >0.$$ In fact, if $r_{t_0}=0 $, then there exists a sequence
    \begin{equation}\label{Set:sequence}
        \{z_i \in C([0,t_0];E), \mathbf{E}^*_L( z_i(0)) \leq C_{r}, \| z_i(t_0)-(m_L-\psi)\Vert_E \geq \beta \}, i \in \mathbb{N},    \end{equation}
    such that \begin{equation}\label{limit:I(u_i)}
        \lim_{i \rightarrow \infty}I_{t_0}(z_i)=0.   
        \end{equation}
    This implies that there exists $f_i$ with $\lim_{i \rightarrow \infty}\|f_i\Vert_{L^2(0,t_0;H)}=0$ such that $z_i$ is the solution of the skeleton equation
   $$ \partial_tz_i(t)= \Delta z_i(t)+ F(z_i(t)) + G(z_i(t))f_i(t).$$
  According to (\ref{ineqn: skeleton}), \begin{equation} \label{u_i bound}
      \sup_{t \geq 0, i \in \mathbb{N}}\|z_i\Vert_E < \infty.  \end{equation} Since the level sets of $\mathbf{E}^*_L(\cdot)$ are compact in $E$ (\cite{bertini2008dobrushin}), there exist $x \in E$ and a subsequence still denoted by $z_i(0)$ such that $$\lim_{i \rightarrow \infty}\|z_i(0)-x\Vert_H \leq L\cdot \lim_{i \rightarrow \infty}\|z_i(0)-x\Vert_E =0.$$  Then, by standard estimates, \begin{align*}
       \frac{1}{2}\frac{d}{dt}\|z_i(t)-&z_0^x(t)\Vert_H^2 \leq \langle\Delta(z_i(t)-z_0^x(t)), z_i(t)-z_0^x(t)\rangle_H \\&+ \langle F(z_i(t))-F(z_0^x(t)), z_i(t)-z_0^x(t)\rangle+ \langle G(z_i(t))f_i(t), z_i(t)-z_0^x(t) \rangle  \\&\ \ \ \ \ \ \ \ \ \ \ \ \  \leq -\|z_i(t)-z_0^x(t) \Vert_{H^1}^2 + C\| z_i(t)-z_0^x(t)\Vert_H^2 + C\|f_i(t)\Vert_H^2.  
       \end{align*} 
       By Gronwall's inequality 
       $$
           \|z_i(t)-z_0^x(t)\Vert_H^2 + 2 \int_0^t \|z_i(s)-z_0^x(s) \Vert_{H^1}^2  \;ds \leq C_t(\|z_i(0)-x\Vert_H^2 + \|f_i\Vert_{L^2([0,t];H)}^2).  $$ 
   Then, there exists $t \in [\frac12t_0, t_0]$ such that
   \begin{equation}\label{ineqn:H^1 dist}
       \|z_i(t)-z_0^x(t) \Vert_{H^1}^2 \leq \frac{C_{t_0}}{t_0}(\|z_i(0)-x\Vert_H^2 + \|f_i\Vert_{L^2(0,t_0;H)}^2).   \end{equation}
       Thanks to Lemma \ref{lemma-energy decreasing}, for any $\eta >0$, one can choose $t_0$ sufficiently large so that $\mathbf{E}^*_L(z_0^x(t)) \leq \frac{\eta}{4} $. Thus, by the Sobolev inequality and (\ref{ineqn:H^1 dist}), there exists $N$ such that for any $ N\leq i \in \mathbb{N} $, $$|\mathbf{E}^*_L(z_i(t))- \mathbf{E}^*_L(z_0^x(t))\vert \leq C\|z_i(t)-z_0^x(t) \Vert_{H^1} \leq \frac{\eta}{4}. $$ Then, 
   $$\mathbf{E}^*_L(z_i(t)) \leq \mathbf{E}^*_L(z_0^x(t)) + \frac{\eta}{4} \leq \frac{\eta}{2}. $$ Furthermore, by (\ref{Energy inequality}) and (\ref{u_i bound}), \begin{align*}
       \mathbf{E}^*_L(z_i(t_0))&\leq \mathbf{E}^*_L(z_i(t)) + C(\sup_{s \geq 0, i \in \mathbb{N}}\|z_i(s)\Vert_E^2+1)\cdot I_{t_0}(z_i) \\ &\leq \mathbf{E}^*_L(z_i(t))+C I_{t_0}(z_i).  \end{align*} By (\ref{limit:I(u_i)}), for any $\eta >0$, there exists $N$ such that for all $i \geq N$, $CI_{t_0}(z_i) \leq \frac{\eta}{2}$.
       Therefore, one obtains that for $i$ sufficiently large, 
       $$ \mathbf{E}^*_L(z_i(t_0)) \leq \eta. $$ Moreover, by Proposition \ref{prop: small energy}, for any $\beta >0$, one can choose $\eta$ sufficiently small such that $$\|z_i(t_0)-(m_L-\psi)\Vert_E \leq \frac{\beta}{2} $$which contradicts the assumption that $ \|z_i(t_0)-(m_L-\psi)\Vert_E  \geq \beta $ in (\ref{Set:sequence}). 
       \end{proof}

    \begin{proof} [Proof of Proposition \ref{prop:compactness of K}]
        Fixing $r\geq 0$, we need to show that for any sequence $\{z_n \}\subset K_{-\infty}(r)$, there exists a subsequence $\{z_{n_k}\}$ converging in $C((-\infty,0];E)$ to some $\hat{z} \in K_{-\infty}(r) $. According to Lemma \ref{lemma:z-energy}, for any $k \in \mathbb{N}$, the restriction of $z_n$ to the interval $[-k,0]$ belongs to $K_{\Lambda,-k,0}$, where $$\Lambda:=\{x \in E; \mathbf{E}^*_L(x) \leq C_r \},$$ and $C_r$ is given in Lemma \ref{lemma:z-energy}. As the level sets of $\mathbf{E}^*_L(\cdot)$ are compact in $E$, we have $K_{\Lambda,-k,0}$ is compact in $C([-k,0];E)$, for any $k \in \mathbb{N}$. 
        
        We choose a subsequence following Helly’s diagonal procedure. Taking $k =1$ we can find $\{z_{n_1}\} \subset \{z_n\}$ and $\hat{z}_1 \in C([-1,0];E)$ such that $z_{n_1|_{[-1,0]}}$ converges to $\hat{z}_1$ in $ C([-1,0];E) $. With the same arguments, we can find a subsequence $\{z_{n_2}\} \subset \{z_{n_1}\}$ and $\hat{z}_2 \in C([-2,0];E)$ such that $z_{n_2|_{[-2,0]}}$ converges to $\hat{z}_2$ in $ C([-2,0];E) $. It is obvious that $\hat{z}_2|_{[-1,0]}=\hat{z}_1$. By this procedure, we can find a subsequence $\{z_{n^\prime}\}\subseteq\{z_n\}$ converging to some $\hat{z}$ in $C((-\infty,0];E)$. Thanks to Theorem 5.1 of \cite{cerrai2004large}, $I_{-k}(\cdot)$ is lower semi-continuous and we have $I_{-k}(\hat{z}) \leq r$ for any $k \geq 0$ which implies $I_{-\infty}(\hat{z}) \leq r$. Moreover, by Lemma \ref{lemma:z-energy}, we have that $$\sup_{t \geq0}\mathbf{E}^*_L(\hat{z}(-t)) \leq C_r.$$ To show that $\hat{z} \in K_{-\infty}(r) $, it remains to prove that
        $$ \lim_{t \rightarrow -\infty}\|\hat{z}(t)-(m_L-\psi)\Vert_E =0. $$ If this does not hold, there exist a constant $\beta>0$ and a sequence $t_n\nearrow +\infty$ such that $ \|\hat{z}(-t_n)-(m_L-\psi)\Vert_E \geq \beta$, for any $n \in \mathbb{N}$. 
       %In fact, if we assume that $ \|\hat{z}(-t_n)-(m_L-\psi)\Vert_E \geq \beta$, 
       Then, by Lemma \ref{lemma : energy increment}, there exists $t_0>0$ such that $I_{-t_n-t_0, -t_n}(\hat{z}) \geq r_{t_0}>0$, for $n \in \mathbb{N}$. Therefore, if we choose a subsequence $\{t_{n_k}\} \subset \{t_n\}$ such that $t_{n_{k+1}} \geq t_{n_k}+t_0$. Then, for any $m \in \mathbb{N}$, we have 
$$I_{-\infty}(\hat{z})\geq \sum_{k=1}^m I_{-t_{n_{k+1}},-t_{n_k}}(\hat{z}) \geq \sum_{k=1}^m I_{-t_{n_{k}}-t_0,-t_{n_k}}(\hat{z})\geq mr_{t_0}. $$
Thus, as $m$ can be taken arbitrarily large, we get $I_{-\infty}(\hat{z})= + \infty$, which contradicts $I_{-\infty}(\hat{z})\leq r $.
\end{proof}

We have the following characterisation of $U_L$. For simplicity of presentation, we denote $$I_{-\infty}^{\inf}(x)=\inf\{I_{-\infty}(z); z \in C((-\infty,0];E), \lim_{t \rightarrow +\infty}\|z(-t)-(m_L-\psi)\Vert_E=0 , z(0)=x \}. $$ If the infimum is attainable, we write $I_{-\infty}^{\inf}(x) $ as $I_{-\infty}^{\min}(x) $.
\begin{proposition}\label{Prop:characterization of U_L}
     Under Assumption \ref{assumption 3}, for any $x \in E $
     $$U_L(x)=I_{-\infty}^{\min}(x). $$
     \end{proposition}
\begin{proof}
    Let $T>0$ be fixed and $z \in C([-T,0];E)$ such that $z(-T)=m_L-\psi$, $z(0)=x $ and $I_{-T}(z) < \infty$. 
	Set 
	$$\bar{z}  (t):=\begin{cases}
		z(t), &t \in [-T ,0], \\ 
		m_L-\psi, &t \in (-\infty,-T].
	\end{cases}$$
	Obviously, $\bar{z} \in C((-\infty,0];E) $, $\bar{z}(0)=x$ and $ \lim _{t \rightarrow +\infty}\|\bar{z}(-t)-(m_L-\psi)\Vert_E=0 $.
    Denote $z_f^{-t_n,m_L-\psi}$ as the solution of equation (\ref{Skeleton equation-2}) with $z_f^{-t_n,m_L-\psi}(-t_n)=m_L-\psi$. Since $m_L-\psi$ is the equilibrium of (\ref{Skeleton equation-2}) with $f=0$, for every $t_0\geq T$, we have $z_0^{-t_n,m_L-\psi}(t)=m_L-\psi$ for every $t \in [-t_n,-T]$.  Thus, we have
    $$I_{-\infty }(\bar{z} ) = I_{-T}(\bar{z}) + I_{-\infty,-T}(\bar{z}) = I_{-T}(z).$$ In particular, this implies that
	$$ I_{-\infty}^{\inf}(x) \leq I_{-T}(z).$$
	Taking the infimum over all $z \in C([-T,0];E)$ and $T>0$, we get
	$$I_{-\infty}^{\inf}(x) \leq U_L(x ).$$
    
Next, we will prove the opposite. If $$I_{-\infty}^{\inf}(x)=+\infty ,$$ there is nothing to prove. Therefore, we only need to consider $z \in K_{-\infty}(r)$, $z(0)=x$, for some $r>0$. 
	For such $z$, there exists a sequence $t_n \nearrow +\infty, n \in \mathbb{N}$ such that $$\lim_{n \rightarrow+\infty}\|z(-t_n)-(m_L-\psi)\Vert_H \leq L\cdot\left(\lim_{n \rightarrow+\infty}\|z(-t_n)-(m_L-\psi)\Vert_E\right) =0.$$ and $$f\in L^2(-\infty,0;H),\ \sup_{n \in \mathbb{N}}I_{-t_n} = \frac{1}{2} \| f\Vert_{L^2(-\infty,0;H) }^2\leq r,$$ where $z$ is the solution of equation (\ref{Skeleton equation-2}), starting from time $-t_n$. Since $f\in L^2(-\infty,0;H) $, for any $\varepsilon >0$,  there exists $T_\varepsilon\leq 0  $ such that $\frac12\|f\Vert_{L^2(-\infty,T_\varepsilon;H)}^2 \leq \varepsilon $.
    Then, by (\ref{ineqn:H^1 dist}), for $n \in \mathbb{N}$ sufficiently large there exists $T_\varepsilon \in [-t_{n},-t_n+1]$ such that $$\|z(T_\varepsilon)-(m_L-\psi)\Vert_{H^1}^2 \leq C(\|z(-t_n)-(m_L-\psi)\Vert_H^2 + 2\varepsilon) \leq C\varepsilon .$$ According to discussions in Proposition \ref{prop:energy lower bound}, there exists a path $v_\varepsilon$ such that 
    $v_\varepsilon (T_\varepsilon-1 )=m_L-\psi$, $v_\varepsilon (T_\varepsilon)=z(T_\varepsilon) $ and $I_{T_\varepsilon-1,T_\varepsilon}(v_\varepsilon)\leq C\epsilon$. We set $$\bar{v}_\varepsilon  (t):=\begin{cases}
		v_\varepsilon(t), &t \in [T_\varepsilon-1 ,T_\varepsilon], \\ 
		z(t), &t \in [T_\varepsilon,0]. 
	\end{cases}$$	Then, $\bar{v}_\varepsilon(T_\varepsilon-1)=m_L-\psi$, $\bar{v}_\varepsilon(0)=x $, and 
    $$I_{T_\varepsilon-1 } (\bar{v}_\varepsilon )\leq I_{-\infty} (z) + C\varepsilon .$$ By the definition of $U_L$,
	$$ U_L(x) \leq I_{T_\varepsilon-1 } (\bar{v}_\varepsilon ).$$ This then implies $$U_L(x ) \leq I_{-\infty}(z) + \varepsilon.$$ Therefore, by taking the infimum over $ \{z \in C((-\infty,0];E), \lim_{t \rightarrow +\infty}\|z(-t)-(m_L-\psi)\Vert_E=0 , z(0)=x \} $, and since $\epsilon>0$ can be arbitrarily small, we get $$U_L(x) \leq  I_{-\infty}^{\inf}(x).$$ 

	Finally, by Proposition \ref{prop:compactness of K}, $K_{-\infty}(r)$ is compact. Hence, $I_{-\infty}^{\inf}(x) $ is indeed $I_{-\infty}^{\min}(x)$. The proof is complete.
	 %$$U_L(x)=\min\{I_{-\infty}(z); z \in C((-\infty,0];E), \lim_{t \rightarrow +\infty}\|z(-t)-(m_L-\psi)\Vert_E=0 , z(0)=x \}. $$
    \end{proof}
    The characterisation of $U_L$ allows us to enhance the compactness of the level sets of $I_{-\infty}(\cdot)$ on $C((-\infty,0];E)$ to the compactness of the level sets of $U_L(\cdot)$ on $E$. This is the main result of this section and a key result in the theory of large deviations for invariant measures.
    \begin{theorem}
         Under Assumption \ref{assumption 3}, for any $r\geq 0$ the level set $$\Phi(r):=\{x \in E: U_L(x) \leq r \}$$ is compact in $E$.
        \end{theorem}
  \begin{proof}
       Let $\{x_n\}\subset \Phi_L(r)$. According to Proposition \ref{Prop:characterization of U_L}, for each $n \in \mathbb{N}$ we can find $z_n \in C((-\infty,0];E)$ with $z_n(0)=x_n$ and $$\lim_{t \rightarrow +\infty}\|z_n(-t)-(m_L-\psi)\Vert_E=0,$$ such that $U_L(x_n)=I_{-\infty}(z_n)$. Since $U_L(x_n)\leq r$, we have
$\{z_n\} \subset K_{-\infty}(r)$. By Proposition \ref{prop:compactness of K}, $K_{-\infty}(r) $ is compact. So, there exists a subsequence $\{z_{n_k}\} \subset \{z_n\}$ such that $z_{n_k}\rightarrow z \in K_{-\infty}(r) $ in $C((-\infty,0];E)$. In particular, $x_{n_k}=z_{n_k}(0) \rightarrow z(0)$ in $E$. Thanks to Proposition \ref{Prop:characterization of U_L}, we have $U_L(z(0)) \leq I_{-\infty}(z)\leq r$. That states
$z(0) \in \Phi_L(r)$. The proof is complete.
\end{proof}

\section{LDP for invariant measures}
In this section, as the main results of this paper, we will prove the LDP for $\{\mu_\epsilon\}_{\epsilon>0}$ under the sublinear growth assumption. This is indicated by the LDP of $\{\bar{\mu}_{\epsilon}\}_{\epsilon>0}$, whose relation with $\{\mu_\epsilon\}_{\epsilon>0}$ is given by (\ref{mu_epsilon}).

\subsection{Exponential estimates}
The essential part of the proof is an exponential estimate for $\bar{\mu}_\epsilon$. 
Choose an integer \(p^\star>6\). We then choose \(k^\star\) such that
\[
\frac{1}{p^\star}<k^\star<
\frac12-\frac{2}{p^\star}.
\]
In particular,
\[
0<k^\star<\frac12,
\qquad
k^\star p^\star>1.
\]
We may therefore choose \(\alpha\) to satisfy.
\[
\frac{k^\star}{2}+\frac{1}{p^\star}
<
\alpha
<
\frac14.
\]
Equivalently,
\begin{equation}\label{ineq:alpha}
    (\alpha-1-\frac{k^\star}{2})\frac{p^\star}{p^\star-1}>-1.\end{equation}
By the Sobolev embedding theorem, $W^{k^\star,p^\star}(-L,L)$ is compactly embedded into $E$. Rewrite equation (\ref{ACE2}) in a different form for $\bar{u}_\epsilon(0)=x$ and $ \lambda>0$
    \begin{equation}\label{rewrite}
       \partial_t \bar{u}_\epsilon^x(t,\xi) = (\Delta-\lambda) \bar{u}_\epsilon^x(t,\xi)  + F(\bar{u}_\epsilon^x(t))(\xi) + \lambda\bar{u}_\epsilon^x(t,\xi) +\epsilon^{\frac{1}{2}} G( \bar{u}_\epsilon^x(t))\partial_t W(t,\xi).
    \end{equation}
    The mild solution $\bar{u}_\epsilon^x(t) $ is represented as
	\begin{align*}
		 \bar{u}_\epsilon ^x(t)= e^{-\lambda t}S(t)x+\int_{0}^{t} S(t-s)e^{-\lambda(t-s)}[F(\bar{u}_\epsilon ^x(s))+\lambda \bar{u}_\epsilon^x(s)] \,ds  + \epsilon^{\frac{1}{2}}\gamma_\lambda(\bar{u}_\epsilon^x)(t),
	\end{align*}
    where, $$\gamma_\lambda(\bar{u}_\epsilon^x)(t):= \int_{0}^{t} S(t-s)e^{-\lambda(t-s)}G(\bar{u}_\epsilon^x(s))\,dW(s) . $$
We first derive the finite-time exponential estimates by means of the Boué–Dupuis variational representation. This approach allows us to treat the stochastic convolution with a linearly growing multiplicative coefficient and to obtain the exponential control required for the solution on every fixed finite time interval.
\begin{lemma}\label{lemma:exponential estimate linear}
           For every $t>0$, $r >0$ and $x \in E$, there exists $\bar{R}_r$ such that
    \begin{equation}\label{inequ: lemma exponential 2}
      \mathbb{P}  \left( \sup_{s \in[0,t]} \|\bar{u}_\epsilon ^x(s)\Vert_{E} \geq \bar{R}_r\right) \leq \exp(-\frac{r}{\epsilon}), \;0<\epsilon \leq 1,
    \end{equation} 
for some $C_t>0$ that depends on $t$.
      Moreover, there exist $\epsilon_r >0$ and $R_r >0$ such that for any $0<t^\prime\leq t$
    \begin{equation}\label{inequ: lemma exponential}
      \sup_{\|x\Vert_E \leq R} \mathbb{P}  \left(  \sup_{s \in[t^\prime,t]} \|\bar{u}_\epsilon ^x(s)\Vert_{W^{k^\star,p^\star}}\geq R_r \right) \leq \exp(-\frac{r}{\epsilon}), \epsilon<\epsilon_r.
    \end{equation}

    \end{lemma}
\begin{proof}
According to the mild form of the equation, \begin{equation}\label{align:section 5}
        \begin{split}    
&\ \ \ \ \|\bar{u}_\epsilon ^x(t)\Vert_{W^{k^\star,p^\star}}=\|(-\Delta)^{\frac{k^\star}{2}}\bar{u}_\epsilon ^x(t)\Vert_{L^{p^\star}}\\ &\leq \int_{0}^{t}\left\| (-\Delta)^{\frac{k^\star}{2}} S(t-s)e^{-\lambda(t-s)}[F(\bar{u}_\epsilon ^x(s))+\lambda \bar{u}_\epsilon^x(s)]\right\Vert_{L^{p^\star} } \,ds  \\&\ \ \ \ + \epsilon^{\frac{1}{2}}\|\gamma_\lambda(\bar{u}_\epsilon^x)(t)\Vert_{W^{k^\star,p^\star}} +  \|(-\Delta)^{\frac{k^\star}{2}}S(t)x\Vert_{L^{p^\star}} \\ &\leq C\int_0^{t} (t-s)^{-\frac{k^\star}{2}}e^{-\lambda(t-s)}\;ds\cdot \sup_{s \in [0,t]}\|F(\bar{u}_\epsilon ^x(s))+\lambda \bar{u}_\epsilon^x(s)\Vert_{L^{p^\star}} \\&\ \ \ \ + \epsilon^{\frac{1}{2}}  \|\gamma_\lambda(\bar{u}_\epsilon^x)(t)\Vert_{W^{k^\star,p^\star}} + C t^{-\frac{k^\star}{2}} \|  x\Vert_{L^{p^\star}} \\& 
\leq C\sup_{s \in [0,t]} \|F(\bar{u}_\epsilon ^x(s))+\lambda \bar{u}_\epsilon^x(s)\Vert_{L^{p^\star}} + \epsilon^{\frac{1}{2}}  \|\gamma_\lambda(\bar{u}_\epsilon^x)(t)\Vert_{W^{k^\star,p^\star}} + C t^{-\frac{k^\star}{2}} \|  x\Vert_{L^{p^\star}} \\ &\leq C\Bigg( \sup_{s \in [0,t]} \|Y_\lambda(\bar{u}_\epsilon ^x )(s)\Vert_{L^{3p^\star}}^3 +\epsilon^{\frac{3}{2}} \sup_{s \in [0,t]}  \|\gamma_\lambda(\bar{u}_\epsilon^x)(s)\Vert_{L^{3p^\star}} ^3\\&\ \ + \epsilon^{\frac{1}{2}}  \|\gamma_\lambda(\bar{u}_\epsilon^x)(t)\Vert_{W^{k^\star,p^\star}} + C t^{-\frac{k^\star}{2}} \|  x\Vert_{L^{p^\star}} +1 \Bigg).  
\end{split}
\end{equation}    
 Here, 
        \begin{equation}\label{def-Y_lambda}
            Y_\lambda(\bar{u}_\epsilon ^x )(t):= \bar{u}_\epsilon ^x(t) - \epsilon^{\frac{1}{2}}\gamma_\lambda(\bar{u}_\epsilon^x)(t)
            \end{equation}
Thus, by (\ref{align:section 5}) and the Sobolev inequality 
\begin{align}
   &\ \ \ \  \mathbb{P}( \sup_{s \in[t^\prime,t]} \|\bar{u}_\epsilon ^x(s)\Vert_{W^{k^\star,p^\star}}\geq R_r )\nonumber \\ \nonumber &\leq \mathbb{P} \left(  \sup_{s\in [0,t]} \|Y_\lambda(\bar{u}_\epsilon ^x )(s)\Vert_{L^{3p^\star}} \geq R_{r,1}\right) +  \mathbb{P}  \left(  \epsilon^{\frac{1}{2}} \sup_{s \in[0,t]} \|\gamma_\lambda(\bar{u}_\epsilon^x)(s)\Vert_{L^{3p^\star}} \geq R_{r,2}\right)  \\ \nonumber &\ \ \ \ + \mathbb{P}  \left(  \epsilon^{\frac{1}{2}} \sup_{s \in [0,t]} \|\gamma_\lambda(\bar{u}_\epsilon^x)(t)\Vert_{W^{k^\star,p^\star}} \geq R_{r,3} \right)\\ 
    %&\ \ \ \ +\mathbb{P}  \left(  \|\bar{u}_\epsilon ^0(t)\Vert_{E} \geq \frac{R_r}{4C} \right)\\ 
    &\leq \mathbb{P}  \left( \sup_{s \in [0,t]} \|\bar{u}_\epsilon ^x(t)\Vert_{E} \geq \tilde{R}_r\right)  +3\mathbb{P}  \left(  \epsilon^{\frac{1}{2}} \sup_{s \in[0,t]} \|\gamma_\lambda(\bar{u}_\epsilon^x)(s)\Vert_{W^{k^\star,p^\star}} \geq \tilde{R}_r\right). %\\ &\ \ \ \ +   \mathbb{P}  \left(  \epsilon^{\frac{1}{2}} \sup_{s \in [0,t]} \|\gamma_\lambda(\bar{u}_\epsilon^0)(s)\Vert_{W^{k^\star,p^\star}} \geq \left(\frac{R_r}{3C} \right)^{\frac{1}{3}}\right)\\ &\ \ \ \  + \mathbb{P}  \left(  \epsilon^{\frac{1}{2}}  \|\gamma_\lambda(\bar{u}_\epsilon^0)(t+1)\Vert_{W^{k^\star,p^\star}} \geq \frac{R_r}{3C} \right)\\ & \leq 3 \mathbb{P}  \left(  \epsilon^{\frac{1}{2}} \sup_{s \in [0,t]} \|\gamma_\lambda(\bar{u}_\epsilon^0)(s)\Vert_{W^{k^\star,p^\star}} \geq \tilde{R}_r \right). 
    \label{eqn:split est}
    \end{align}
    Here, $R_{r,1},R_{r,2}, R_{r,3},\tilde{R}_r>0$ are constants derived from $R_r$.

   To connect the linear-growth case with the bounded-noise estimate, we use a stopping argument. For \(R>0\), define the stopping time
\[
\tau_R(v)
:=
\inf\left\{
s\geq 0:
\| v(s)\|_E\geq R
\right\}
\wedge t
\]
and the stopped stochastic convolution
\[
\gamma_{\lambda,R}(\bar u_\epsilon^x)(t)
:=
\int_0^t
S(t-s)e^{-\lambda(t-s)}
\mathbf 1_{\{s\leq \tau_R(u_\epsilon^x )\}}
G(\bar u_\epsilon^x(s))\,\mathrm dW(s),
\qquad t\in[0,T].
\]
On the event \(\{\tau_R(u_\epsilon^x )=t\}\), we have
\[
\gamma_{\lambda,R}(\bar u_\epsilon^x)(s)
=
\gamma_\lambda(\bar u_\epsilon^x)(s),
\qquad s\in[0,t].
\]
Consequently, for every \(a>0\), $\bar{R}>0$
\begin{align}
&\mathbb{P}\left(
    \sqrt{\epsilon}\,
    \sup_{s\in[0,t]}
    \left\|
        \gamma_\lambda(\bar u_\epsilon^x)(s)
    \right\|_{W^{k^\star,p^\star}}
    \geq R_r
\right)
\nonumber\\
&\qquad\leq
\mathbb{P}\left(\tau_{\bar R}( u_\epsilon^x )\leq t\right)
+
\mathbb{P}\left(
    \sqrt{\epsilon}\,
    \sup_{s\in[0,t]}
    \left\|
        \gamma_{\lambda,\bar R}
        (\bar u_\epsilon^x)(s)
    \right\|_{W^{k^\star,p^\star}}
    \geq R_r
\right)
\nonumber\\
&\qquad\leq
\mathbb{P}\left(
    \sup_{s\in[0,t]}
    \|\bar u_\epsilon^x(s)\|_E
    \geq \bar R
\right)
+
\mathbb{P}\left(
    \sqrt{\epsilon}\,
    \sup_{s\in[0,t]}
    \left\|
        \gamma_{\lambda,\bar R}
        (\bar u_\epsilon^x)(s)
    \right\|_{W^{k^\star,p^\star}}
    \geq R_r
\right)
\nonumber\\
&\qquad=: I_1(t)+I_2(t).
\label{eq:localisation-stochastic-convolution}
\end{align}
Since \(G\) has at most linear growth, there exists \(C>0\) such that
\[
\left\|
\mathbf 1_{\{s\leq\tau_{\bar{R}}(u_\epsilon^x ) \}}
G(\bar u_\epsilon^x(s))
\right\|_{L^\infty}
\leq C(1+\bar{R}),
\qquad s\in[0,t].
\]
Therefore, the integrand in the stopped stochastic convolution is a bounded predictable process. 

We first estimate $I_1(t)$. We will show that there exists $\bar{R}>0$ such that $$\mathbb{P}  \left(   \sup_{s \in [0,t]} \|\bar{u}_\epsilon^x(s)\Vert_E \geq \bar{R}_r \right) \leq \exp\left( -\frac{r}{\epsilon}\right),0<\epsilon \le1. $$ Notice that by (\ref{ineqn: skeleton}), the $C([0,t];E)$ norm of $z_f^x$ is bounded by $\|f\Vert_{L^2(0,t;H)}$ for all $t\geq0$.
Then, for any $r>0$, the level set $K^x_t(r)$ is uniformly bounded in $C([0,t];E)$ with respect to all $t \geq0$. 
Therefore, suppose we can prove that for any $r>0$ and all $t \geq 0$, there exists $\epsilon_r>0$ that does not depend on $t$, such that for any $\epsilon< \epsilon_r$
$$\mathbb{P}\left(\text{dist}_{C([0,t];E)}(\bar{u}_\epsilon^x,K^x_t(2r))\geq M \right) \leq \exp({-\frac{r}{\epsilon}}),$$ then the proof is complete. Here, $M>0$ depends on $r$ and will be chosen later. To show that, we define functions from $C([0,t];E)$ to $\mathbb{R}$,
\begin{equation}\label{enq:test function}
    h_{r,t,M}(\varphi):=2r-2r \cdot\min\left\{1, \frac{\text{dist}_{C([0,t];E)}(\varphi,K^x_t(2r)) }{M}\right\}, \; \varphi \in C([0,t];E).
\end{equation}
By this definition, it is clear that $h_{r,t,M} $ is bounded and Lipschitz continuous, with a Lipschitz constant of $\frac{2r}{M}$. Moreover, by the measurability of $\bar{u}^x_\epsilon $ on $C([0,t];E)$, the real valued random variable $h_{r,t,M}(\bar{u}^x_\epsilon) $ is also measurable.
Note that if $\text{dist}_{C([0,t];E)}(\bar{u}^x_\epsilon,K^x_t(2r))\geq M $, then $h_{r,t,M}(\bar{u}^x_\epsilon)=0 $. Thus,
\begin{equation}\label{inequality:h}
    \mathbb{P}\left(\text{dist}_{C([0,t];E)}(\bar{u}_\epsilon^x,K^x_t(2r))\geq M \right) \leq \mathbb{E}\exp\left(-\frac{h_{r,t,M} (\bar{u}^x_\epsilon) }{\epsilon} \right).
\end{equation}
Set $\mathcal{P}_2$ to be the collection of $\mathcal{F}_t$-adapted $H$ valued processes $f(t)$ satisfying 
$$\mathbb{P}(\|f\Vert_{L^2([0,t];H)}< +\infty)=1,$$ and set $\mathcal{P}_2^N$ to be the collection of $\mathcal{F}_t$-adapted $H$ valued processes $f(t)$ with 
$$\mathbb{P}(\frac{1}{2}\|f\Vert_{L^2([0,T];H)}^2\leq N)=1.$$
Define by $\bar{u}^x_{\epsilon,f}$ the dynamics of the following equation on $E$
\begin{equation}\label{eqn: u epsilon f}
    \begin{cases}
        \partial_t\bar{u}_{\epsilon,f}^x = \Delta \bar{u}_{\epsilon,f}^x(t) + F( \bar{u}_{\epsilon,f}^x (t)) + G( \bar{u}_{\epsilon,f}^x (t))f(t)+ \epsilon^{\frac12}G( \bar{u}_{\epsilon,f}^x (t))\partial_tW^L(t) ,\\
       \bar{u}_{\epsilon,f}^x (0,\xi)= x.\\ 
    \end{cases}
\end{equation}
Then, by the variational representation in Theorem 3 of \cite{budhiraja2008large}, for any $\epsilon>0$
\begin{equation}\label{eqn:variation 1}
\begin{split}
   \epsilon \log \mathbb{E}\exp\left(-\frac{h_{r,t,M} (\bar{u}^x_\epsilon) }{\epsilon} \right)= -\inf_{f \in \mathcal{P}_2}\mathbb{E}\left\{ \frac12\int_0^t\|f(s)\Vert_H^2\;ds+ h_{r,t,M}(\bar{u}^x_{\epsilon,f}) \right\}.
     \end{split}
\end{equation}
By the definition of infimum, for any $\eta>0$, there exists $f_n \in \mathcal{P}_2$ such that 
\begin{equation}\label{inequality: variation}
\begin{split}
    &\ \ \ \  \mathbb{E}\left\{ \frac12\int_0^t\|f_n(s)\Vert_H^2\;ds+ h_{r,t,M}(\bar{u}^x_{\epsilon,f_n}) \right\}\\&\leq \inf_{f \in \mathcal{P}_2}\mathbb{E}\left\{ \frac12\int_0^t\|f(s)\Vert_H^2\;ds+ h_{r,t,M}(\bar{u}^x_{\epsilon,f}) \right\} +\eta.
\end{split}    
    \end{equation}
We claim that such $f_n$ must satisfy $\mathbb{E}\left[ \frac12\int_0^t\|f_n(s)\Vert_H^2\;ds\right]\leq 2r+\eta $. This is due to the fact that $h_{r,t,M} $ is bounded by $2r$. Indeed, if $\mathbb{E}\left[\frac12\int_0^t\|f_n(s)\Vert_H^2\;ds\right]> 2r+\eta $, then by (\ref{inequality: variation}), one gets
$$\inf_{f \in P_2}\mathbb{E}\left\{ \frac12\int_0^t\|f(s)\Vert_H^2\;ds+ h_{r,t,M}(\bar{u}^x_{\epsilon,f}) \right\} > 2r. $$ However, by taking $f=0$, one gets
$$\inf_{f \in \mathcal{P}_2}\mathbb{E}\left\{ \frac12\int_0^t\|f(s)\Vert_H^2\;ds+ h_{r,t,M}(\bar{u}^x_{\epsilon,f}) \right\} \leq 2r, $$ which contradicts. Furthermore, by the localisation arguments in Theorem 4.4 of \cite{budhiraja2000variational}, there exists $f_n \in \mathcal{P}_2^N$ such that
\begin{equation}\label{inequality:variation 1}
\begin{split}
    &\ \ \ \  \mathbb{E}\left\{ \frac12\int_0^t\|f_n(s)\Vert_H^2\;ds+ h_{r,t,M}(\bar{u}^x_{\epsilon,f_n}) \right\}\\&\leq \inf_{f \in \mathcal{P}_2}\mathbb{E}\left\{ \frac12\int_0^t\|f(s)\Vert_H^2\;ds+ h_{r,t,M}(\bar{u}^x_{\epsilon,f}) \right\} +\eta+\frac{4r(2r+\eta)}{N}.
\end{split}    
    \end{equation}
Similarly, by the definition of the rate function
\begin{equation}\label{inequality:variation 2}
\begin{split}
    \inf_{\varphi \in C([0,t];E) }\{I_t^x(\varphi)+ h_{r,t,M}(\varphi) \} &= \inf_{f \in L^2(0,t;H)}\left\{ \frac12\int_0^t\|f(s)\Vert_H^2\;ds+ h_{r,t,M}(z^x_f) \right\}\\ &\leq \mathbb{E} \left\{\frac12\int_0^t\|f_n(s)\Vert_H^2\;ds+ h_{r,t,M}(z^x_{f_n})\right\}.
     \end{split}
\end{equation}
By substituting (\ref{inequality:variation 1}) and (\ref{inequality:variation 2}) into (\ref{eqn:variation 1}) and setting $N=16r$, we get 
\begin{align*}
    &\ \ \ \ \epsilon \log \mathbb{E}\exp\left(-\frac{h_{r,t,M} (\bar{u}^x_\epsilon) }{\epsilon} \right)+  \inf_{\varphi \in C([0,t];E) }\{I_t^x(\varphi)+ h_{r,t,M}(\varphi) \}\\ &\leq 
    -\mathbb{E}\left\{ \frac12\int_0^t\|f_n(s)\Vert_H^2\;ds+ h_{r,t,M}(\bar{u}^x_{\epsilon,f_n}) \right\} + \mathbb{E}\left\{\frac12\int_0^t\|f_n(s)\Vert_H^2\;ds+ h_{r,t,M}(z^x_{f_n})\right\} +\eta+ \frac{4r(2r+\eta)}{N} \\ &\leq \mathbb{E}\left[h_{r,t,M}(z^x_{f_n})- h_{r,t,M}(\bar{u}^x_{\epsilon,f_n}) \right]+\frac{5}{4}\eta+\frac{r}{2} \\ &\leq  \frac{2r}{M}\mathbb{E}\|z^x_{f_n}-\bar{u}^x_{\epsilon,f_n} \Vert_{C([0,t];E)}+\frac{5}{4}\eta+\frac{r}{2}. \end{align*}
Here, the last inequality is due to the Lipschitz continuity of $h_{r,t,M} $. Since $f_n \in \mathcal{P}_2^N $, as in the proof of Theorem 5.5 of \cite{cerrai2003stochastic}(see also Theorem 2.2 and Theorem 4.1 of \cite{cerrai2004large}), we have that there exists $C_{t,N}>0$ that depends on $r$ and $t$ such that for any $0<\epsilon\leq1$ and 
\begin{equation}\label{inequality: expectation}
 \mathbb{E}\left[\sup_{0\leq s \leq t}\|z^x_{f_n}(s)\Vert_E\right]\leq C_{t,N}(\|x\Vert_E+1), \;\mathbb{E}\left[\sup_{0\leq s \leq t}\| \bar{u}^x_{\epsilon,f_n}(s) \Vert \right] \leq C_{t,N}(\|x\Vert_E+1 ) .   
\end{equation}
Hence, by choosing $M\geq 16C_{t,N}(\|x\Vert_E+1) $ and $0<\eta \leq \frac{r}{5}$, one gets that for any $0<\epsilon\leq1$ and all $t\geq 0$
\begin{equation}\label{inequality:variation 3}
\begin{split}
   &\ \ \ \ \epsilon \log \mathbb{E}\exp\left(-\frac{h_{r,t,M} (\bar{u}^x_\epsilon) }{\epsilon} \right)+  \inf_{\varphi \in C([0,t];E) }\{I_t^x(\varphi)+ h_{r,t,M}(\varphi) \}\\ &\leq \frac{\mathbb{E}\left[\sup_{0\leq s \leq t}\|z^x_{f_n}(s)\Vert_E\right] + \mathbb{E}\left[\sup_{0\leq s \leq t}\| \bar{u}^x_{\epsilon,f_n}(s) \Vert \right] }{M}\cdot r+ \frac{r}{2} + \frac{5}{4}\eta \leq r. 
\end{split}
\end{equation}
Note that for any $\varphi \in C([0,t];E) $, either $\varphi \in K^x_t(2r) $ in which case $h_{r,t,M}(\varphi)=2r $, or
$\varphi \notin K^x_t(2r) $ in which case $I_t^x(\varphi) >2r$. This implies that
\begin{equation}\label{inequality: I+h}
    \inf_{\varphi \in C([0,t];E) }\{I_t^x(\varphi)+ h_{r,t,M}(\varphi) \} \geq 2r.
\end{equation}
Now, substituting (\ref{inequality:variation 3}), (\ref{inequality: I+h}) and (\ref{eqn:variation 1}) into (\ref{inequality:h}) , we get for any $0<\epsilon\leq1$ and all $t\geq 0$
\begin{equation}\label{eqn:exponential estimate u}
   \begin{split}
      &\ \ \ \ \epsilon\log\mathbb{P}\left(\text{dist}_{C([0,t];E)}(\bar{u}_\epsilon^x,K^x_t(2r))\geq M \right) \\&\leq \epsilon\log\mathbb{E}\exp\left(-\frac{h_{r,t,M} (\bar{u}^x_\epsilon) }{\epsilon} \right) + \inf_{\varphi \in C([0,t];E) }\{I_t^x(\varphi)+ h_{r,t,M}(\varphi) \}-2r \\ &\leq -r. 
   \end{split} 
\end{equation}
Moreover, by the pathwise a priori estimate for the controlled
skeleton equation,
\[
\sup_{\varphi\in K_t^x(2r)}
\|\varphi\|_{C([0,t];E)}
\leq
C_{t,r}
\left(
1+\|x\|_E
\right).
\]
Therefore,
\[
\left\{
\|\bar u_\epsilon^x\|_{C([0,t];E)}
\geq
C_{t,r}
\left(
1+\|x\|_E
\right)+M
\right\}
\subseteq
\left\{
\operatorname{dist}_{C([0,t];E)}
\left(
\bar u_\epsilon^x,K_t^x(2r)
\right)
\geq M
\right\}.
\]
Indeed, if
\[
\operatorname{dist}_{C([0,t];E)}
\left(
\bar u_\epsilon^x,K_t^x(2r)
\right)<M,
\]
then there exists \(\varphi\in K_t^x(2r)\) such that
\[
\|\bar u_\epsilon^x-\varphi\|_{C([0,t];E)}<M,
\]
and hence
\[
\|\bar u_\epsilon^x\|_{C([0,t];E)}
<
M+
C_{t,r}
\left(
1+\|x\|_E
\right).
\]
It follows from \eqref{eqn:exponential estimate u} that
\[
\mathbb P
\left(
\|\bar u_\epsilon^x\|_{C([0,t];E)}
\geq
C_{t,r}
\left(
1+\|x\|_E
\right)+M
\right)
\leq
\exp\left(
-\frac{r}{\epsilon}
\right),
\qquad
0<\epsilon\leq1.
\]
By (\ref{inequality: expectation}) and (\ref{eqn:exponential estimate u}), \(M\) is chosen to be of order
\(C_{t,r}(1+\|x\|_E)\), then
$$\mathbb{P}  \left(   \sup_{s \in [0,t]} \|\bar{u}_\epsilon^x(s)\Vert_E \geq C_{t,r}(\|x\Vert_E+1)  \right)  \leq \exp\left(-\frac{r}{\epsilon}\right),\; 0<\epsilon \leq 1.$$ Therefore, (\ref{inequ: lemma exponential 2}) follows by taking $\bar{R}_r \geq C_{t,r}(\|x\Vert_E+1) $.

We next estimate \(I_2(t)\) by repeating the variational argument used
for \(I_1(t)\). 
By the factorisation formula, $$\gamma_{\lambda,\bar{R}_r}(\bar u_\epsilon^x )(t)=C_\alpha\int_0^t (t-s)^{\alpha-1} S(t-s)e^{-\lambda(t-s)}\Gamma^\alpha(\bar u_\epsilon^x(s) )(s)\;ds, $$ where $$\Gamma^\alpha(\bar u_\epsilon^x )(s)= \int_0^s (s-r)^{-\alpha}S(s-r)e^{-\lambda(s-r)} 1_{\{r\leq \tau_{\bar{R}}\}}
G(\bar u_\epsilon^x(r))
 \;dW(r).$$
                   Then, by (\ref{ineq:alpha})
                    \begin{equation}\label{ineqn:Gamma(R_r)}
                    \begin{split}
                        \|\gamma_{\lambda,R}(\bar u_\epsilon^x )(t)\Vert_{W^{k^\star,p^\star}} ^{p^\star} &\leq C_{\alpha}^{p^\star}\left(\int_0^t (t-s)^{\alpha-1-\frac{k^\star}{2}}e^{-\lambda(t-s)}\|\Gamma^{\alpha}(\bar u_\epsilon^x )(s)\Vert_{L^{p^\star}}\;ds\right)^{p^\star} \\& \leq C_{\alpha}^{p^\star} \left(\int_0^t (t-s)^{(\alpha-1-\frac{k^\star}{2})\frac{p^\star}{p^\star-1}}e^{-\lambda(t-s)\frac{p^\star}{p^\star-1}}\;ds \right)^{(p^\star-1)} \\ &\ \ \ \ \times \left(\int_0^t  \|\Gamma^{\alpha}(\bar u_\epsilon^x )(s)\Vert_{L^{p^\star}}^{p^\star}\;ds\right)  \\ &\leq C_{\alpha, k^\star,p^\star} \int_0^t  \|\Gamma^{\alpha}(\bar u_\epsilon^x )(s)\Vert_{L^{p^\star}}^{p^\star}\;ds                        %\times \left(\int_0^t e^{-\frac{\lambda}{2}(t-s)p^\star} \;ds\right) \times  \sup_{s \in[0,t]}\|\Gamma^{\alpha}(v )(s)\Vert_{L^{p^\star}}^{p^\star}\\ & \leq  C\sup_{s \in[0,t]}\|\Gamma^{\alpha}(v )(s)\Vert_{L^{p^\star}}^{p^\star} .            
                    \end{split}
                        \end{equation}  
Denote
\begin{equation}\label{eqn-e_k}
e_k(\xi)
=
\begin{cases}
\displaystyle
\frac{\sin\left(\frac{k\pi\xi}{2L}\right)}{\sqrt{L}},
& k\in\mathbb{N}\ \text{is even},\\[2mm]
\displaystyle
\frac{\cos\left(\frac{k\pi\xi}{2L}\right)}{\sqrt{L}},
& k\in\mathbb{N}\ \text{is odd}.
\end{cases}
\end{equation}
Then $\{e_k\}_{k\in\mathbb N}$ is an orthonormal basis of $H$, and
\[
S(s)e_k=e^{-\sigma_k s}e_k,
\qquad
\sigma_k=\left(\frac{k\pi}{2L}\right)^2.
\]
Let $p_s(\xi,\eta)$ be the Dirichlet heat kernel associated with
$S(s)$. Then
\[
p_s(\xi,\eta)
=
\sum_{j\in\mathbb N}
e^{-\sigma_j s}e_j(\xi)e_j(\eta).
\]
Before the stopping time $\tau_{\bar R_r}$,
\[
\left\|G(u_\epsilon^x(s))\right\|_{L^\infty(-L,L)}
\leq C(1+\bar R_r).
\]
Hence, by It\^o's isometry and Parseval's identity, for every
$t\geq0$ and $\xi\in[-L,L]$,
\begin{align*}
\mathbb E
\left|
\Gamma^\alpha(u_\epsilon^x)(t,\xi)
\right|^2
&\leq
\int_0^t
s^{-2\alpha}e^{-2\sigma s}
\sum_{k\in\mathbb N}
\left|
S(s)\left(
G(u_\epsilon^x(t-s))e_k
\right)(\xi)
\right|^2
\,\mathrm ds
\\
&=
\int_0^t
s^{-2\alpha}e^{-2\sigma s}
\sum_{k\in\mathbb N}
\left|
\int_{-L}^{L}
p_s(\xi,\eta)
G(u_\epsilon^x(t-s,\eta))
e_k(\eta)
\,\mathrm d\eta
\right|^2
\,\mathrm ds
\\
&=
\int_0^t
s^{-2\alpha}e^{-2\sigma s}
\int_{-L}^{L}
p_s(\xi,\eta)^2
\left|
G(u_\epsilon^x(t-s,\eta))
\right|^2
\,\mathrm d\eta
\,\mathrm ds
\\
&\leq
C(1+\bar R_r)^2
\int_0^t
s^{-2\alpha}e^{-2\sigma s}
\int_{-L}^{L}
p_s(\xi,\eta)^2
\,\mathrm d\eta
\,\mathrm ds.
\end{align*}
By the semigroup property,
\[
\int_{-L}^{L}p_s(\xi,\eta)^2\,\mathrm d\eta
=
p_{2s}(\xi,\xi)
=
\sum_{j\in\mathbb N}
e^{-2\sigma_j s}e_j(\xi)^2.
\]
Since $|e_j(\xi)|\leq L^{-1/2}$,
\begin{align*}
p_{2s}(\xi,\xi)
&\leq
\frac{1}{L}
\sum_{j\in\mathbb N}
\exp\left(
-\frac{j^2\pi^2}{2L^2}s
\right)
\\
&\leq
Cs^{-1/2}.
\end{align*}
Therefore,
\begin{align*}
\mathbb E
\left|
\Gamma^\alpha(u_\epsilon^x)(t,\xi)
\right|^2
&\leq
C(1+\bar R_r)^2
\int_0^t
s^{-2\alpha-\frac12}e^{-2\sigma s}
\,\mathrm ds
\\
&\leq
C_{\alpha,\lambda}(1+\bar R_r)^2,
\end{align*}
provided that $\alpha<1/4$.
  \begin{comment}         
            Thus, for any $\|v\Vert_{C([0,t];E)} \leq \bar{R}_r$, \begin{equation}\label{ineqn:Gamma(R_r)}
                \sup_{s \in [0,t]}\|\gamma_\lambda(v)(s)\Vert_{W^{k^\star,p^\star}} \leq C \left(\int_0^t \left\|\Gamma^\alpha_{\bar{R}_r}(s) \right\Vert_{L^{p^\star}}^{p^\star}\;ds \right)^{1/p^\star}          \end{equation}
                where $$\Gamma^\alpha_{\bar{R}_r}(t)=\int_0^t (t-s)^{-\alpha}S(t-s)e^{-\lambda(t-s)}(1+\bar{R}_r )\;dW(s). $$

 Denote 
     \begin{equation}\label{eqn-e_k}
   e_k(\xi)= \begin{cases}
     \frac{\sin(\frac{k\pi\xi}{2L})}{\sqrt{L}}, \; k \in \mathbb{N} \;\text{and is even} \\
     \frac{\cos(\frac{k\pi\xi}{2L})}{\sqrt{L}}, \; k \in \mathbb{N} \;\text{and is odd.}     \end{cases}
     \end{equation}  
     Clearly, $e_k$ is the orthogonal basis of $H$. 
     Thus, for any $t \geq 0$, $\xi \in [-L,L]$ \begin{align*}
         \mathbb{E}|\Gamma^\alpha(u_\epsilon^x )(t,\xi)\vert^2 & \leq \sum_{k \in \mathbb{N}}  \int_0^t s^{-2\alpha} e^{-2(\frac{L^{-1}k\pi}{2})^2s} (1+ \bar{R}_r)^2\;ds \cdot e_k^2(\xi) \\& \leq C \sum_{k \in \mathbb{N}} \frac{1}{(L^{-1}k\pi)^{2(1-2\alpha)}}\cdot(1+ \bar{R}_r)^2\\ & \leq C_\alpha (1+ \bar{R}_r)^2 ,        \end{align*} 
\end{comment}
Here, $C_\alpha>0 $ depends on $\alpha$ and does not depend on $t$. For fixed \(s\) and \(\xi\), \(\Gamma^\alpha(\bar u_\epsilon^x)(s,\xi)\) is the terminal value of a square-integrable martingale; the BDG inequality applies, so       \begin{align*}
             \mathbb{E}  \left[\int_0^t\|\Gamma^\alpha(v)(s)\Vert_{L^{p^\star}}^{p^\star}\;ds\right]  &\leq \int_0^t \mathbb{E} \int_{-L}^L |\Gamma^\alpha(v) (s,\xi)\vert^{p^\star}\;d \xi\;ds  \\ &\leq   \int_{-L}^L t\cdot C_{p^\star}(C_\alpha (1+ \bar{R}_r)^2)^\frac{p^\star}{2}\;d \xi  \\ &\leq C_{p^\star,t}C_{\alpha}^{\frac{p^\star}{2}} (1 + \bar{R}_r)^{p^\star}  . 
             \end{align*}

For \(f\in L^2(0,t;H)\), let \(z_f^x\) denote the solution of the
controlled skeleton equation and define
\[
\begin{aligned}
    \Gamma_{t,\bar R}^{\alpha,x}(f)(\tau)
    &:=
    \int_0^\tau
        (\tau-s)^{-\alpha}
        S(\tau-s)e^{-\lambda(\tau-s)}
        \mathbf 1_{\{s\leq\tau_{\bar R}(z_f^x)\}}
        G(z_f^x(s))f(s)\,ds.
\end{aligned}
\]
For \(r>0\), set
\[
    K_{t,\bar R}^{\alpha,x}(r)
    :=
    \left\{
        \Gamma_{t,\bar R}^{\alpha,x}(f):
        \frac12\int_0^t\|f(s)\|_H^2\,ds\leq r
    \right\}.
\]
Notice that the notation
\[
    I_{t,\bar R}^{\alpha,x}(z)
    :=
    \frac12
    \inf\left\{
        \int_0^t\|f(s)\|_H^2\,ds:
        z=\Gamma_{t,\bar R}^{\alpha,x}(f)
    \right\}
\]
may be introduced for convenience, but no large deviation principle
for \(\Gamma^\alpha\) is used below.
\begin{comment}
Let \(\mathcal P_2\) denote the collection of predictable \(H\)-valued
processes \(f\) such that
\[
    \int_0^t\|f(s)\|_H^2\,ds<\infty
    \qquad\mathbb P\text{-a.s.},
\]
and let
\[
    \mathcal P_2^N
    :=
    \left\{
        f\in\mathcal P_2:
        \frac12\int_0^t\|f(s)\|_H^2\,ds\leq N
        \quad\mathbb P\text{-a.s.}
    \right\}.
\]
\end{comment}
For \(f\in\mathcal P_2\), let \(\bar u_{\epsilon,f}^x\) be the
controlled stochastic solution and define
\[
\begin{aligned}
    \Gamma_{\epsilon,f,\bar R}^{\alpha,x}(\tau)
    &:=
    \sqrt{\epsilon}
    \int_0^\tau
        (\tau-s)^{-\alpha}
        S(\tau-s)e^{-\lambda(\tau-s)}
        \mathbf 1_{\{s\leq
        \tau_{\bar R}(\bar u_{\epsilon,f}^x)\}}
        G(\bar u_{\epsilon,f}^x(s))\,dW(s)
\\
    &\quad+
    \int_0^\tau
        (\tau-s)^{-\alpha}
        S(\tau-s)e^{-\lambda(\tau-s)}
        \mathbf 1_{\{s\leq
        \tau_{\bar R}(\bar u_{\epsilon,f}^x)\}}
        G(\bar u_{\epsilon,f}^x(s))f(s)\,ds.
\end{aligned}
\]
We claim that, for every \(N>0\),
\begin{equation}\label{eqn:controlled-Gamma-convergence}
    \sup_{f\in\mathcal P_2^N}
    \mathbb E
    \left\|
        \Gamma_{\epsilon,f,\bar R}^{\alpha,x}
        -
        \Gamma_{t,\bar R}^{\alpha,x}(f)
    \right\|_{L^{p^\star}(0,t;L^{p^\star})}
    < \infty, 0<\epsilon\leq 1.
\end{equation}
Indeed, the stochastic part is bounded by the BDG inequality,
while the controlled part are bounded due to 
\[
\left\|
\mathbf 1_{\{s\leq\tau_{\bar{R}}(\bar u_\epsilon^x )\}}
G(\bar u_\epsilon^x(s))
\right\|_{L^\infty}, \left\|
\mathbf 1_{\{s\leq\tau_{\bar{R}}(\bar u_{\epsilon,f}^x )\}}
G(\bar u_{\epsilon,f}^x(s))
\right\|_{L^\infty}\leq C(1+\bar{R}),
\qquad s\in[0,s],
\] and the corresponding
factorisation estimates.

Moreover,
\[
\sup_{z\in K_{t,\bar R}^{\alpha,x}(2r)}
\|z\|_{L^{p^\star}(0,t;L^{p^\star})}
\leq
C_{t,\alpha,p^\star,\lambda}
(1+\bar R)\sqrt r.
\] 
Therefore, $K_{t,\bar R}^{\alpha,x}(r)$ is bounded in
\(L^{p^\star}(0,t;L^{p^\star})\).  
Then, by the same variational argument as the estimates of $I_1$
%and Theorem 3.49 of \cite{hairer2009introduction}, 
            and (\ref{ineqn:Gamma(R_r)}), there exist $ \delta^\star >0$ and $\epsilon_r$, $\tilde{R}_r$ that depend on $\bar{R}_r$, $r$ such that 
            \begin{align*}
                &\ \ \ \ \mathbb{P}  \left( \epsilon^{\frac{1}{2}} \sup_{s \in [0,t]} \|\gamma_{\lambda, \bar{R}}(u_\epsilon^x )(s)\Vert_{W^{k^\star,p^\star}}\geq \tilde{R}_r  \right)\\ &\leq \mathbb{P}  \left( \epsilon^{\frac{1}{2}} \|\Gamma^\alpha(\bar{u}^x_\epsilon)\Vert_{L^{p^\star}(0,t;L^{p^\star})}\geq \tilde{R}_r/C \right)\\ &\leq  \mathbb{P} \left(\epsilon^{\frac{1}{2}}\Gamma^\alpha(\bar{u}^x_\epsilon) \in B_{L^{p^\star}([0,t];L^{p^\star})}^c(K_{t,\bar R}^{\alpha,x}(r),M_r)   \right)\\&\leq \exp(-\frac{r}{\epsilon}),\; \epsilon\leq \epsilon_r.            \end{align*}
            %and $$\Psi(r):=\{\mathcal{W} \in W^{1,2}([0,+\infty);H): \frac{1}{2}\int_0^{+\infty}\|\partial_t\mathcal{W}(t)\Vert_H^2\;dt \leq r\}.$$ Here, the second inequality is due to the fact that for $f \in L^2([0,+\infty);H) $ such that $\frac12\|f(\cdot)\Vert_{L^2([0,+\infty);H)}^2 \leq r $, $\mathcal{W} \in C $ and  $$ $$ of rate functional $I^v$.
            Hence, for $\epsilon\leq\epsilon_r$
            \begin{equation}\label{inequ: exponential 1}
                \mathbb{P}  \left(  \epsilon^{\frac{1}{2}} \sup_{s\in[0,t]} \|\gamma_\lambda(\bar{u}_\epsilon^x)(s)\Vert_{W^{k^\star,p^\star}} \geq \tilde{R}_r,\;  \sup_{s \in [0,t]}\|\bar{u}_\epsilon^x(s)\Vert_E \leq \bar{R}_r \right)\leq \exp(-\frac{r}{\epsilon}),
                \end{equation}
which is the desired result. Then, the proof is finished by (\ref{eqn:split est}) and (\ref{eq:localisation-stochastic-convolution}).              
 
\end{proof}    
Combining the preceding variational estimates with the absorbing structure of $\bar{u}^x_\epsilon$ and an iteration argument, we obtain the following exponential tail bound for the invariant measures.
\begin{proposition}\label{prop:exponential estimate}
Assume further
\[
    |g(\xi,\sigma)|
    \leq C\bigl(1+|\sigma|^\rho\bigr),
    \qquad
    (\xi,\sigma)\in[-L,L]\times\mathbb R,
    \qquad 0\leq\rho<1.
\]   Then, for any $r >0$, there exist $\epsilon_r >0$ and $R_r >0$ such that 
    \begin{equation}
        \bar{\mu}_\epsilon(B_{E}^c(0,R_r))\leq \exp(-\frac{r}{\epsilon}), \;\epsilon < \epsilon_r.
    \end{equation}
\end{proposition}
\begin{proof}

Set the controlled process for $\gamma_\lambda(\bar{u}_\epsilon^x) $ as
\[
\begin{aligned}
\mathcal{Z}_{\epsilon,f}(\tau)
&:=
\epsilon^{1/2}
\int_0^\tau
S(\tau-s)e^{-\lambda(\tau-s)}
G(\bar{u}_{\epsilon,f}^x(s))\,dW(s) \\
&\quad+
\int_0^\tau
S(\tau-s)e^{-\lambda(\tau-s)}
G(\bar{u}_{\epsilon,f}^x(s))f(s)\,ds.
\end{aligned}
\]
The corresponding skeleton is
\[
\mathcal Z_f(\tau)
=
\int_0^\tau
S(\tau-s)e^{-\lambda(\tau-s)}
G(\bar u_f^x(s))f(s)\,ds.
\]
To treat $\mathcal{Z}_{\epsilon,f} $ and $\mathcal Z_f $ simultaneously, for
\(\theta\in[0,1]\) and \(f\in\mathcal P_2^N\), let \(u_{\theta,f}\)
denote the corresponding controlled solution and set
\[
\mathcal Z_{\theta,f}(s)
:=
\theta
\int_0^s
S(s-r)G(u_{\theta,f}(r))\,dW(r)
+
\int_0^s
S(s-r)G(u_{\theta,f}(r))f(r)\,dr.
\]
Thus,
\[
u_{\sqrt{\epsilon},f}=u_{\epsilon,f},
\qquad
\mathcal Z_{\sqrt{\epsilon},f}
=
\mathcal Z_{\epsilon,f},
\]
whereas
\[
u_{0,f}=u_f,
\qquad
\mathcal Z_{0,f}
=
\mathcal Z_f.
\]
Notice that \(u_f\) and \(\mathcal Z_f\) may still be random, since
\(f\in\mathcal P_2^N\) is allowed to be a random predictable control.
By Remark 4.3 of \cite{cerrai2003stochastic}, together with the
corresponding factorisation estimate for the controlled convolution
(see also \cite[Eq.~(4.2)]{cerrai2004large}), for sufficiently large
\(p\), we have
\begin{align}
\mathbb E
\left\|
\mathcal Z_{\theta,f}
\right\|_{C([0,t];E)}^p
&\leq
C_{t,p}\theta^p
\left(
1+
\mathbb E
\left\|
u_{\theta,f}
\right\|_{C([0,t];E)}^{\rho p}
\right)
\nonumber\\
&\quad+
C_{t,p}
\mathbb E
\left[
\|f\|_{L^2(0,t;H)}^p
\left(
1+
\left\|
u_{\theta,f}
\right\|_{C([0,t];E)}^{\rho p}
\right)
\right]
\nonumber\\
&\leq
C_{t,p}
\left(
1+N^{p/2}
\right)
\left(
1+
\mathbb E
\left\|
u_{\theta,f}
\right\|_{C([0,t];E)}^{\rho p}
\right),
\label{eq:unified-controlled-convolution-moment}
\end{align}
uniformly in \(\theta\in[0,1]\), where we have used
\[
\|f\|_{L^2(0,t;H)}
\leq
\sqrt{2N}
\qquad\text{a.s.}
\]
Consequently,
\begin{align}
\mathbb E
\left\|
\mathcal Z_{\theta,f}
\right\|_{C([0,t];E)}
&\leq
\left(
\mathbb E
\left\|
\mathcal Z_{\theta,f}
\right\|_{C([0,t];E)}^p
\right)^{1/p}
\nonumber\\
&\leq
C_{t,p}
\left(
1+N^{1/2}
\right)
\left[
1+
\left(
\mathbb E
\left\|
u_{\theta,f}
\right\|_{C([0,t];E)}^{\rho p}
\right)^{1/p}
\right].
\label{eq:unified-controlled-convolution-first-moment}
\end{align}
Set
\[
\beta:=1-\rho\in(0,1),
\]
and, for \(\theta\in[0,1]\) and \(f\in\mathcal P_2^N\), define
\[
X_{\theta,f}
:=
\left(
\mathbb E
\left\|
u_{\theta,f}
\right\|_{C([0,t];E)}^p
\right)^{1/p}.
\]
Following the a priori argument in the proof of Theorem 5.5 of
\cite{cerrai2003stochastic}, the pathwise estimate for the equation
forced by \(\mathcal Z_{\theta,f}\) gives
\[
X_{\theta,f}
\leq
C_{t,p}
\left[
1+\|x\|_E+
\left(
\mathbb E
\left\|
\mathcal Z_{\theta,f}
\right\|_{C([0,t];E)}^p
\right)^{1/p}
\right].
\]
By \eqref{eq:unified-controlled-convolution-moment} and
\[
\left(
\mathbb E
\left\|
u_{\theta,f}
\right\|_{C([0,t];E)}^{\rho p}
\right)^{1/p}
\leq
X_{\theta,f}^{\rho},
\]
we obtain
\[
X_{\theta,f}
\leq
C_{t,p}
\left(
1+\|x\|_E+N^{1/2}
\right)
+
C_{t,p}
\left(
1+N^{1/2}
\right)
X_{\theta,f}^{\rho}.
\]
By Young's inequality, for every \(\delta>0\),
\[
C_{t,p}
\left(
1+N^{1/2}
\right)
X_{\theta,f}^{\rho}
\leq
\delta X_{\theta,f}
+
C_{t,p,\beta,\delta}
\left(
1+N^{1/(2\beta)}
\right).
\]
Choosing \(\delta>0\) sufficiently small and absorbing
\(\delta X_{\theta,f}\) into the left-hand side, we conclude that
\[
\sup_{\theta\in[0,1]}
\sup_{f\in\mathcal P_2^N}
X_{\theta,f}
\leq
C_{t,p,\beta}
\left(
1+\|x\|_E+N^{1/(2\beta)}
\right).
\]
It follows that
\[
\sup_{\theta\in[0,1]}
\sup_{f\in\mathcal P_2^N}
\left(
\mathbb E
\left\|
u_{\theta,f}
\right\|_{C([0,t];E)}^{\rho p}
\right)^{1/p}
\leq
C_{t,p,\beta}
\left(
1+\|x\|_E^\rho
+N^{\rho/(2\beta)}
\right).
\]
Combining this estimate with
\eqref{eq:unified-controlled-convolution-first-moment}, we obtain
\[
\mathbb E
\left\|
\mathcal Z_{\theta,f}
\right\|_{C([0,t];E)}
\leq
C_{t,p,\beta}
\left(
1+N^{1/2}
\right)
\left(
1+\|x\|_E^\rho
+N^{\rho/(2\beta)}
\right).
\]
For every \(M>0\), Young's inequality yields
\[
\left(
1+N^{1/2}
\right)
\|x\|_E^\rho
\leq
\frac{\|x\|_E}{M}
+
C_{\beta,M}
\left(
1+N^{1/(2\beta)}
\right).
\]
Moreover,
\[
\left(
1+N^{1/2}
\right)
\left(
1+N^{\rho/(2\beta)}
\right)
\leq
C_\beta
\left(
1+N^{1/(2\beta)}
\right),
\]
since
\[
\frac{1}{2}+\frac{\rho}{2\beta}
=
\frac{1}{2\beta}.
\]
Therefore,
\[
\sup_{\theta\in[0,1]}
\sup_{f\in\mathcal P_2^N}
\mathbb E
\left\|
\mathcal Z_{\theta,f}
\right\|_{C([0,t];E)}
\leq
C_{t,\beta,M}
\left(
1+N^{1/(2\beta)}
\right)
+
\frac{\|x\|_E}{M}.
\]
In particular,
\[
\sup_{0<\epsilon\leq1}
\sup_{f\in\mathcal P_2^N}
\mathbb E
\left\|
\mathcal Z_{\epsilon,f}
\right\|_{C([0,t];E)}
\leq
C_{t,\beta,M}
\left(
1+N^{1/(2\beta)}
\right)
+
\frac{\|x\|_E}{M},
\]
and
\[
\sup_{f\in\mathcal P_2^N}
\mathbb E
\left\|
\mathcal Z_f
\right\|_{C([0,t];E)}
\leq
C_{t,\beta,M}
\left(
1+N^{1/(2\beta)}
\right)
+
\frac{\|x\|_E}{M}.
\]
Here \(C_{t,\beta,M}>0\) is independent of
\(\theta\), \(\epsilon\), \(x\), and \(f\in\mathcal P_2^N\).
\begin{comment}
Combining this uniform
estimate with the factorisation estimate for the controlled stochastic
convolution, the sublinear growth condition
and Young's inequality, we obtain, for every $t>0$, $M>0$, and
$\beta:=1-\rho\in(0,1]$,
\[
    \sup_{f\in\mathcal{P}_N}
    \mathbb{E}
    \left\|
        \mathcal{Z}_{\epsilon,f}
    \right\|_{C([0,t];E)}
    \leq
    C_{t,\beta,M}
    \left(
        1+N^{\frac{1}{2\beta}}
    \right)
    +
    \frac{\|x\|_E}{M},
    \qquad 0<\epsilon\leq 1.
\]

The constant $C_{t,\beta,M}$ is independent of
$\epsilon$, $x$, and $f\in\mathcal{P}_N$.
Likewise, for the corresponding controlled skeleton, we have
\[
    \sup_{\frac12\|f\|_{L^2(0,t;H)}^2\leq N}
    \left\|
        \mathcal{Z}_{f}
    \right\|_{C([0,t];E)}
    \leq
    C_{t,\beta,M}
    \left(
        1+N^{\frac{1}{2\beta}}
    \right)
    +
    \frac{\|x\|_E}{M}.
\]
Here, \(C_{t,\beta,M}>0\) is independent of \(\epsilon\), \(x\), and \(f\).
\end{comment}
Therefore, by repeating the preceding variational argument, for every $M>0$ there exists
\(C_{t,M}>0\), independent of \(r\) and \(\epsilon\), such that
\begin{equation}\label{eqn: prob gamma}
   \mathbb{P}  \left(   \sup_{s \in [0,t]} \|\epsilon^{1/2}\gamma_\lambda(\bar{u}_\epsilon^x)(s)\Vert_E \geq C_{t,\beta,M}(r^{\frac{1}{2\beta}}+1) +\|x\Vert_E /M \right)  \leq \exp\left(-\frac{r}{\epsilon}\right),\;0<\epsilon\leq 1.
\end{equation}
Proceeding as in the proof of Theorem 5.5 of
\cite{cerrai2003stochastic}, we apply the subdifferential inequality
for the $E$-norm to
\[
    \bar u_\epsilon^x
    -
    \epsilon^{1/2}\gamma_\lambda(\bar u_\epsilon^x)
\]
and then use the scalar comparison principle. Consequently, there
exist constants $\kappa,C>0$ such that
\[
    \|\bar u_\epsilon^x(t)\|_E
    \leq
    C e^{-\kappa t}\|x\|_E
    +
    C\sup_{s\in[0,t]}
    \left\|
        \epsilon^{1/2}
        \gamma_\lambda(\bar u_\epsilon^x)(s)
    \right\|_E
    +
    C_\lambda,
    \qquad t\geq0.
\]
By the flow property, regarding the process on $[nt,(n+1)t]$ as starting from $\bar u_\epsilon^0(nt)$ and being driven by shifted Brownian motion, we have
$$\|\bar{u}_\epsilon^x((n+1)t)\Vert_E \leq Ce^{-\kappa t}\|\bar{u}_\epsilon^x(nt) \Vert_E +  C \sup_{s \in [0,t]} \|\epsilon^{1/2}\gamma_\lambda(\bar{u}_\epsilon^{\bar{u}_\epsilon^x(nt)})(s)\Vert_E +C_\lambda. $$
Then, we can choose $C_{\lambda,t,\beta,M}$ large enough such that
\begin{align*}
  \|\bar{u}_\epsilon^x((n+1)t)\Vert_E-C_{\lambda,t,\beta,M} &\leq Ce^{-\kappa t}(\|\bar{u}_\epsilon^x(nt)\Vert_E-C_{\lambda,t,\beta,M})\\  &\ \ \ + C (\sup_{s \in [0   ,t]} \|\epsilon^{1/2}\gamma_\lambda(\bar{u}_\epsilon^{\bar{u}_\epsilon^x(nt)})(s)\Vert_E -C_{t,\beta,M})  
\end{align*}
Setting $X(n)= \|\bar{u}_\epsilon^0(nt)\Vert_E-C_{\lambda,t,\beta,M} $ and $Z(n)=C (\sup_{s \in [0   ,t]} \|\epsilon^{1/2}\gamma_\lambda(\bar{u}_\epsilon^{\bar{u}_\epsilon^0(nt)})(s)\Vert_E -C_{t,\beta,M}) $ and taking $t$ sufficiently large such that $\theta:=C e^{-\kappa t}<1 $. 
Then, \begin{equation}\label{eqn:iteration 1}
    X(n+1) \leq \theta X(n) +Z(n).
\end{equation}
Let $1<B<\frac{1}{\theta} $. Taking fixed $M$ such that it satisfies
$$B\theta+\frac{CB}{M}=:q< 1.$$ 
%For any $r>0$, taking $R$ sufficiently large such that $$RBe^{-\kappa t}+ \frac{CBR}{M}+C_{t,M}C\cdot r<R.$$
Then, according to $(\ref{eqn:iteration 1})$, for any $R>0$, $$\{X(n)<BR\}\bigcap\Big\{ Z(n)<R(1- B\theta-\frac{CB}{M} )+\frac{CX(n)}{M}\Big \} \Rightarrow \{X(n+1)<R \} .$$
Thus, \begin{equation}\label{eqn:iteration 2}
\begin{split}
    \mathbb{P}(X(n+1)\geq R)&\leq \mathbb{P}(X(n)\geq BR) \\&\ \ + \mathbb{P}\left(Z(n)\geq\frac{C(X(n)+C_{\lambda,t,\beta,M})}{M}+  (1-q)R-\frac{CC_{\lambda,t,\beta,M}}{M} \right).
    \end{split}
    \end{equation}
By (\ref{eqn: prob gamma}),
\begin{align*}
   &\ \ \ \ \mathbb{P}\left(Z(n)\geq\frac{C(X(n)+C_{\lambda,t,\beta,M})}{M}+  (1-q)R -\frac{CC_{\lambda,t,\beta,M}}{M} \right)\\ & = \mathbb{P}  \left(   \sup_{s \in [0,t]} \|\epsilon^{1/2}\gamma_\lambda(\bar{u}_\epsilon^{\bar{u}_\epsilon^x(nt) })(s)\Vert_E \geq C_{t,\beta,M} +\|\bar{u}_\epsilon^x(nt) \Vert_E /M  + \frac{(1-q)R}{C}-C_{\lambda,t,\beta,M}/M\right) \\& \leq \exp\left[-\frac{\left(\frac{1-q}{CC_{t,\beta,M}}R- \frac{C_{\lambda,t,\beta,M}}{C_{t,\beta,M}M}\right)^{2\beta}}{\epsilon}\right], \; 0<\epsilon\leq 1.
\end{align*}
Then, by the iteration (\ref{eqn:iteration 2}),
\begin{equation}
    \mathbb{P}(X(n+1)\geq R) \leq \mathbb{P}(X(0)\geq B^{n+1}R)+ \sum_{i=0}^{n+1} \exp\left[-\frac{\left(\frac{1-q}{CC_{t,\beta,M}}B^iR- \frac{C_{\lambda,t,\beta,M}}{C_{t,\beta,M}M}\right)^{2\beta}}{\epsilon}\right].    \end{equation}
Setting $a:= \frac{1-q}{CC_{t,\beta,M}} $ and $b:= \frac{C_{\lambda,t,\beta,M}}{C_{t,\beta,M}M} $, by Bernoulli's inequality $B^{2\beta i} \geq 1+ i(B^{2\beta}-1) $, we have 
\begin{align*}
   & \ \ \ \ \sum_{i=0}^{\infty} \exp\left[-\frac{(aB^{i}R- b)^{2\beta}}{\epsilon}\right] \\&\leq \sum_{i=0}^{\infty} \exp\left[-\frac{(B^{2\beta i}(aR)^{2\beta}- b^{2\beta}}{\epsilon}\right]  \\&\leq  \exp\left(-\frac{(aR)^{2\beta}- b^{2\beta}}{\epsilon}\right) \sum_{i=0}^{\infty}\left[ \exp\left(-\frac{(aR)^{2\beta}(B^{2\beta}-1)}{\epsilon}\right) \right]^i \\&= \frac{\exp\left(-\frac{(aR)^{2\beta}- b^{2\beta}}{\epsilon}\right) }{1-\exp\left(-\frac{(aR)^{2\beta}(B^{2\beta}-1)}{\epsilon}\right) }. \end{align*}
For an arbitrary $s>0$, writing $s=nt+\tau$ with $\tau\in[0,t)$, we apply the above iteration on the intervals $[\tau+kt,\tau+(k+1)t]$ and estimate the remaining initial interval $[0,\tau]$ by the finite-time bound.    
Taking $x=0$, then
  $$\mathbb{P}\left(\| \bar{u}_\epsilon^0(s)\Vert_E \geq R+C_{\lambda,t,,\alpha,M} \right) \leq \frac{\exp\left(-\frac{(aR)^{2\beta}- b^{2\beta}}{\epsilon}\right) }{1- \exp\left(-\frac{(aR)^{2\beta}(B^{2\beta}-1)}{\epsilon}\right) }.   $$
Therefore, by taking $R$ large with $(aR)^{2\beta}-b^{2\beta}>0$, we have for every $s>0$,
\begin{equation}\label{eqn:all time exponential est}
    \mathbb{P}\left(\| \bar{u}_\epsilon^0(s)\Vert_E \geq R_r \right) \leq C\exp \left(-\frac{r}{\epsilon} \right), 0<\epsilon\leq 1.
\end{equation}
Let
\[
\bar\mu_{\epsilon,T_n}
:=
\frac{1}{T_n}\int_0^{T_n}P_t(0,\cdot)\,\mathrm dt\]
By the Krylov–Bogoliubov argument, there exists a sequence
\(T_n\to+\infty\) such that
\[
\bar\mu_{\epsilon,T_n}\rightharpoonup\bar\mu_\epsilon,
\] in $E$.
Since \(\{x\in E:\|x\|_E>R_r\}\) is open, the Portmanteau theorem gives
\begin{align*}
\bar\mu_\epsilon\bigl(\|x\|_E>R_r\bigr)
&\leq
\liminf_{n\to\infty}
\bar\mu_{\epsilon,T_n}\bigl(\|x\|_E>R_r\bigr)
\\
&\leq
C\exp\left(-\frac{2r}{\epsilon}\right).
\end{align*}
Replacing \(R_r\) by a slightly larger radius and decreasing
\(\epsilon_r\), we obtain
\[
\bar\mu_\epsilon\bigl(\|x\|_E\geq R_r\bigr)
\leq
\exp\left(-\frac{r}{\epsilon}\right),
\]
which implies the desired result.

\end{proof}

\begin{theorem}\label{thm:exponential-tightness}
Under the assumption of Proposition \ref{prop:exponential estimate}, for every \(r>0\), there exist \(R_r>0\) and \(\epsilon_r>0\) such that
\[
\bar\mu_\epsilon\left(
B_{W^{k^\star,p^\star}}(0,R_r)^c
\right)
\leq
\exp\left(-\frac{r}{\epsilon}\right),
\qquad
0<\epsilon\leq\epsilon_r.
\]
Consequently, \(\{\bar\mu_\epsilon\}_{\epsilon>0} \) is exponentially tight in \(E\).
\end{theorem}
\begin{proof}
Fix \(r>0\). By the exponential estimate in \(E\), there exists
\(\bar R_r>0\) such that
\[
\bar\mu_\epsilon\left(B_E(0,\bar R_r)^c\right)
\leq
\exp\left(-\frac{2r}{\epsilon}\right).
\]
By the finite-time exponential estimate Lemma \ref{lemma:exponential estimate linear}, for some \(t>0\), there exist
\(R_r>0\) and \(\epsilon_r>0\) such that
\[
\sup_{\|x\|_E\leq\bar R_r}
\mathbb P\left(
\|\bar u_\epsilon^x(t)\|_{W^{k^\star,p^\star}}
\geq R_r
\right)
\leq
\exp\left(-\frac{2r}{\epsilon}\right),
\qquad
0<\epsilon\leq\epsilon_r.
\]
Therefore, by the invariance of \(\bar\mu_\epsilon\),
\begin{align*}
\bar\mu_\epsilon\left(
B_{W^{k^\star,p^\star}}(0,R_r)^c
\right)
&\leq
\bar\mu_\epsilon\left(B_E(0,\bar R_r)^c\right)
\\
&\quad+
\sup_{\|x\|_E\leq\bar R_r}
\mathbb P\left(
\|\bar u_\epsilon^x(t)\|_{W^{k^\star,p^\star}}
\geq R_r
\right)
\\
&\leq
2\exp\left(-\frac{2r}{\epsilon}\right)
\leq
\exp\left(-\frac{r}{\epsilon}\right),
\end{align*}
after decreasing \(\epsilon_r\) if necessary.
\end{proof}

\subsection{The LDP Lower bound}
\begin{proposition}\label{Pro:lower bound}
    Under the assumption of Proposition \ref{prop:exponential estimate}, the family of invariant measures $\{\bar{\mu}^\epsilon\}_{\epsilon>0}$ of equation (\ref{ACE2}) satisfies the LDP lower bound in $E$ with rate functional $U_L(\cdot)$. That is, for any $\zeta \in E,\;\delta>0$ and $ \gamma>0$, there exists $\epsilon_0>0$ such that 
$$\bar{\mu}_\epsilon(B_E(\zeta,\delta)) \geq \exp{\left(-\frac{U(\zeta)+\gamma}{\epsilon} \right)}, \;\epsilon < \epsilon_0.$$
\end{proposition}
\begin{proof}
    Fix $\zeta  \in E$, $\delta >0$, $\gamma >0$, and $T>0$. We assume that $U(\zeta) < \infty$, or otherwise the result would be trivial. We set $\left\{z^x\right\}_{x \in E} \subset C ([0,T];E)$ to be a family of paths satisfying 
    \begin{equation}\label{eqn:6.1}
        \sup_{x \in E}\left\| z^x(T) -\zeta \right\Vert _E <\delta /2.
        \end{equation}
	 %Since the expected value of $\|\bar{u}_\epsilon^x(t)\Vert_E$ is bounded uniformly with respect to $t>0,\; 0<\epsilon<1$( see \cite{cerrai2003stochastic} Proposition 6.1). 
	%By Chebyshev's inequality, we have that for some $x\in E$ and $C>0$ depends on $\|x\Vert_E$,
    %\begin{align*}
	   % \bar{\mu}_\epsilon(B_E^c(R)) &= \lim_{n \rightarrow +\infty}\frac{1}{t_n}\int_{0}^{t_n}\mathbb{P}(\|\bar{u}_\epsilon^x(s)\Vert_E>R)\;ds \\ &\leq \lim_{n \rightarrow +\infty}\frac{1}{t_n}\int_{0}^{t_n} \frac{\mathbb{E}[\|\bar{u}_\epsilon^x(s)\Vert_E]}{R} \;ds \\ &\leq \frac{C}{R}.
        %\end{align*}
   By Theorem \ref{thm:exponential-tightness}, there exists a constant $R_r$ being sufficiently large and independent of $\epsilon$ such that \begin{equation}\label{eqn:6.2}
       \bar{\mu} _\epsilon \left( \overline{B_{W^{k^\star,p^\star} }(0,R_r) }  \right) \geq \frac{1}{2}, \ \text{for all}\ \epsilon_0>\epsilon >0.
       \end{equation}

Moreover, according to Section \ref{sec:LDP Dynamics}, $\left\{\mathcal{L} (\bar{u}_\epsilon^x)\right\}_{\epsilon > 0}$  satisfies the Freidlin-Wentzell uniform large deviations principle in $C([0,T];E)$. So, for every $r_0,\gamma >0$ there exists $\epsilon _0>0$ such that $\gamma /2 + \epsilon_0 \ln \frac{1}{2} >0$ and for any $z^x \in C([0,T];E)$ with $I_T^x(z^x) \leq r_0$,
   \begin{equation}\label{eqn:6.3}
       \inf _{x \in B_E(0,R)}\mathbb{P} \left(\left\| \bar{u}_\epsilon ^x -z^x\right\Vert _{C([0,T];E)} < \delta /2 \right) \geq \inf _{x \in B_E(0,R)}\exp\left(-\frac{1}{\epsilon }\left[I_T^x(z^x) + \gamma /2 + \epsilon \ln \frac{1}{2}\right]\right),   \end{equation}
for every $\epsilon \leq \epsilon _0$. 
Since the family  $\left\{\bar{\mu} _\epsilon \right\}_{\epsilon >0}$ is the invariant measures of equation (\ref{ACE2}), from (\ref{eqn:6.1}), (\ref{eqn:6.2}) and (\ref{eqn:6.3}), we have 
	\begin{align*}
		\bar{\mu} _\epsilon \left(B_E(\zeta ,\delta )\right) &= \int_{E} \mathbb{P} \left(\| \bar{u}_\epsilon ^x(T) -\zeta \Vert _E <\delta \right) \,d\bar{\mu} _\epsilon (x ) \\ &\geq \int_{E} \mathbb{P} \left(\left\| \bar{u}_\epsilon ^x -z^x\right\Vert _{C([0,T];E)} < \delta /2 \right) \,d\bar{\mu} _\epsilon (x )\\ &\geq \int_{\overline{B_{W^{k^\star,p^\star} }(0,R) }  } \mathbb{P} \left(\left\| \bar{u}_\epsilon ^x -z^x\right\Vert _{C([0,T];E)} < \delta /2 \right) \,d\bar{\mu} _\epsilon (x )\\ &\geq \bar{\mu} _\epsilon \left(\overline{B_{W^{k^\star,p^\star} }(0,R) } \right) \inf _{x \in \overline{B_{W^{k^\star,p^\star} }(0,R) } }\mathbb{P} \left(\left\| \bar{u}_\epsilon ^x -z^x\right\Vert _{C([0,T];E)} < \delta /2 \right) \\ &\geq \exp\left(\frac{\epsilon \ln\frac{1}{2}}{\epsilon }\right)\inf _{x \in \overline{B_{W^{k^\star,p^\star} }(0,R) } }\exp\left(-\frac{1}{\epsilon }\left[I_T^x(z^x) + \gamma /2 + \epsilon \ln \frac{1}{2}\right]\right) \\ &= \inf _{x \in \overline{B_{W^{k^\star,p^\star} }(0,R) } }\exp\left(-\frac{1}{\epsilon }\left[I_T^x(z^x) + \gamma /2 \right]\right),
	\end{align*}
	for every $\epsilon \leq \epsilon _0$. Therefore, to complete the proof, it suffices to find a sufficiently large $T$ such that for each $x \in \overline {B_{W^{k^\star,p^\star} }(0,R)}  $, there exists a path $z^x \in C([0,T];E)$ with $z^x(0) =x$ that satisfies $$I_T(z^x) \leq U(\zeta) + \gamma /2$$ and $$\| z^x(T) - \zeta \Vert _E < \delta /2.$$ 
    For this, denote the solution of equation (\ref{Skeleton equation-2}) of the skeleton function by $z^x_f$. By the definition of $U(\zeta )$, there exist $T_2 > 0$, $z_{\bar{f}}^{m_L-\psi} \in C([0,T];E)$, and $\bar{f} \in L^2([0,T];H)$ such that $z^{m_L-\psi}_{\bar{f}}(0)=m_L-\psi$, $z^{m_L-\psi}_{\bar{f}}(T_2)=\zeta $, and $$\frac{1}{2}\int_{0}^{T_2} \| \bar{f}(t)\Vert_H^2  \,dt = I_{T_2}( z^{m_L-\psi}_{\bar{f}}) \leq U(\zeta ) + \gamma /2.$$ Meanwhile, by Lemma \ref{lemma-energy decreasing}, when $f =0$, we have that for any $x \in B_E(r)$, $$ \lim_{t \rightarrow +\infty}\| z^x_0(t)-(m_L-\psi)\Vert _{E} \leq C\lim_{t \rightarrow \infty}\| z^x_0(t)-(m_L-\psi)\Vert_{H^1} = 0.$$ Thus, by the compactness of $\overline{B_{W^{k^\star,p^\star} }(0,R) } $ in $E$ and (\ref{z_phi^x-z_varphi^y}), for any $\eta>0 $, we can take $T_1= T_1(\eta )$ large enough such that 
	$$ \sup_{x \in\overline{B_{W^{k^\star,p^\star} }(0,R) }  }\left\| z^x_0(T_1)-(m_L-\psi)\right\Vert _E \leq \eta.$$ 
	We choose the paths $z^x=z^x_f$ with  $f \in L^2([0,T];H)$ defined by
	$$f(t):=\begin{cases}
		0, &t \in [0 ,T_1), \\ 
		\bar{f} (t-T_1), &t \in [T_1,T_1+T_2],
	\end{cases}$$
	and $T= T_1 + T_2$. Clearly, we have \begin{align*}
		I_T(z^x) &= \frac{1}{2}\int_{0}^{T} \| f(t)\Vert_H^2  \,dt =\frac{1}{2}\int_{0}^{T_2} \| \bar{f}(t)\Vert_H^2  \,dt  \leq U(\zeta ) + \gamma /2.
	\end{align*}
	Moreover, from the above construction and (\ref{z_phi^x-z_varphi^y}) we have that
	\begin{align*}
		\sup_{x \in \overline{B_{W^{k^\star,p^\star} }(0,R) }}\left\| z^x_f(T) - \zeta \right\Vert _E &\leq \sup_{y \in B_E(m_L-\psi,\eta ) }\left\| z^y_{\bar{f}}(T_2) - z^{m_L-\psi}_{\bar{f}}(T_2)\right\Vert_E  \\ &\leq \sup_{y \in B_E(m_L-\psi,\eta )}c(T,\| y\Vert _E,U(\zeta))\cdot \| y-(m_L-\psi)\Vert_E  .
	\end{align*}
	Therefore, we can choose $\eta $ small enough such that $$\| z^x_f(T) -\zeta \Vert_E < \delta /2.$$ In conclusion, the path $z^x(t) = z^x_f(t)$ satisfies the two desired conditions. The proof is complete.

    \end{proof}
\subsection{The LDP upper bound}
To prove the upper bound for LDP, we need the following two auxiliary results.

\begin{lemma}\label{2 upper bound lemma 1}
	 Under Assumption \ref{assumption 3}, for any $\delta >0$ and $r>0$, there exist $\beta >0$ and $T>0$ such that for any $t^\star \geq T$ and $z \in C([-t^\star,0];E)$,
	$$\|z(-t^\star)-(m_L-\psi)\Vert_E <\beta , \ I_{t^\star}(z) \leq r \Rightarrow dist_E \left(z(0), \Phi_L (r)\right) <\delta,$$ where $$\Phi_L (r) := \left\{x  \in E: U_L(x )\leq r\right\}.$$
\end{lemma}
\begin{proof}
    Assume the claim does not hold. Then, there exist $\delta>0$, $r>0$ and a sequence of functions $z_n\in C([-T_n,0];E)$ with $T_n \nearrow +\infty$ such that $\beta_n := \|z_n(-T_n)-(m_L-\psi)\Vert_E\searrow 0$ and
\begin{equation}\label{2 ineq 1 upper lemma 1}
    I_{T_n}(z_n) \leq r \ , dist_E(z_n(0), \Phi_L(r)) \geq \delta.
\end{equation}

Thus, $z_n$ also satisfies equation (\ref{Skeleton equation-2}) with $z_n(-T_n)=x$ and $f$ such that $\|f \Vert_{L^2([-T_n,0];H)}^2= 2I_{T_n}(z_n) \leq 2r$. %According to the proof of Theorem 5.1 in \cite{cerrai2004large}, for any compact set $\Lambda \in E$ the level set
%$$K_{\Lambda,t_1,t_2}(r):= \{z \in C([t_1,t_2];E), z(t_1)\in \Lambda, I_{t_1,t_2}(z) \leq r\}$$ is compact. %Due to (\ref{ineqn: skeleton}),
%$$\sup_{n \in \mathbb{N},t \in [-T_n,0]}\|z(t)\Vert_E < \infty. $$
%Then, by standard estimates,
%\begin{align*}
       %\frac{1}{2}\frac{d}{dt}\|z_n(t)\Vert_H^2 &\leq \langle\Delta z_n(t), z_n(t)\rangle_H + \langle F(z_n), z_n\rangle+ \langle G(t,z_n(t))f_n(t), z_n \rangle  \\& \leq -\|z_n \Vert_{H^1}^2 - \kappa\| z_n\Vert_H^2 + C\|f_n(t)\Vert_H^2 +C_\kappa.  
       %\end{align*} 
       %By Gronwall's inequality, for $t \in [-T_n,0]$
       %$$\|z_n(t)\Vert_H^2 + 2 \int_{-T_n}^t \|z_n(s) \Vert_{H^1}^2  \;ds \leq Ct(\|z_n(-T_n)\Vert_H^2 + 2r +1).  $$

The following arguments are similar to the proof of Proposition \ref{prop:compactness of K}. By Lemma \ref{lemma:z-energy}, for any $k \in [0,T_n-1]$ and $k \in \mathbb{N}$, the restriction of $z_n$ to the interval $[-k,0]$ belongs to $K_{\Lambda,-k,0}$, where 
$$\Lambda:=\{\bar{u} \in E; \mathbf{E}^*_L(\bar{u}) \leq C_r \},$$ and $C_r$ is a constant depending on $r$.
As the level sets of $\mathbf{E}^*_L(\cdot)$ are compact in $E$, we find that $K_{\Lambda,-k,0}$ is compact in $C([-k,0];E)$. Taking $k =1$, we can find $\{z_{n_1}\} \subseteq \{z_n\}$ and $\hat{z}_1 \in C([-1,0];E)$ such that $z_{n_1|_{[-1,0]}}$ converges to $\hat{z}_1$ in $ C([-1,0];E) $. With the same arguments, we can find a subsequence $\{z_{n_2}\} \subset \{z_{n_1}\}$ and $\hat{z}_2 \in C([-2,0];E)$ such that $z_{n_2|_{[-2,0]}}$ converges to $\hat{z}_2$ in $ C([-2,0];E) $. Continuing this procedure, we can find a subsequence $\{z_{n^\prime}\}\subset\{z_n\}$ converging to some $\hat{z}$ in $C((-\infty,0];E)$. Thanks to Theorem 5.1 of \cite{cerrai2004large}, $I_{-k}(\cdot)$ is lower semi-continuous, so we have $I_{-k}(\hat{z}) \leq r$ for any $k \geq 0$ which implies $I_{-\infty}(\hat{z}) \leq r$. By the same argument as in the proof of Proposition \ref{prop:compactness of K},
we obtain
\[
\lim_{t\to+\infty}
\|\hat z(-t)-(m_L-\psi)\|_E=0.
\] 
Therefore, $\hat{z} \in K_{-\infty}(r)$. Thus, by the characterisation of $U_L$ (see Proposition \ref{Prop:characterization of U_L}), $U_L(\hat{z}(0)) \leq r$ and $\hat{z}(0) \in \Phi_L(r)$. However, by our construction, $$ \lim_{n_1 \rightarrow \infty }\|z_{n_1}(0)-\hat{z}(0)\Vert_E=0, $$ which contradicts (\ref{2 ineq 1 upper lemma 1}). The proof of the desired result is complete.    
           
%By the arguments in Proposition \ref{u infty estimate}, for any $\eta >0$, there exists $N \in \mathbb{N}$ such that for any $n \geq N$ we can find $t_n \geq 0$ with $T_n - t_n$ sufficiently large and satisfies $\|z_n(t_n)\Vert_{H^{\beta}} < \eta$. 
%Denote $k_n=[\sqrt{T_n}]$, then there exists $i \in \mathbb{N}$ and $ i \leq k_n$ such that $$\frac12\int_{\frac{(i-1)T_n}{k_n}}^{ \frac{iT_n}{k_n} } \|f(s)\Vert_H^2 \;ds \leq \frac{r}{k_n}.$$
%According to arguments in Proposition \ref{prop:energy lower bound}, there exists paths $v_n \in C([0,t^\star];E)$ such that $v_n(0)= m_L-\psi$, $v_n(t^\star)= z_n(0)$ and $\lim_{n \rightarrow \infty} I_{t^\star}(v_n)=0$ . This implies that for any $\varepsilon >0$, there exists $N \in \mathbb{N}$ such that for any $n \geq N$,
%$$\limsup_{n \rightarrow \infty} U(z_n(T_n)) \leq \limsup_{n \rightarrow \infty} I_{T_n}(z_n)+ \lim_{n \rightarrow \infty} I_{t^\star}(v_n) \leq r .$$
%Since level sets of $U$ is compact in $E$, there exists a subsequence $(n_k)_{k \in \mathbb{N}}$ and $\zeta \in H $ such that
%$$\lim_{k \rightarrow \infty}\|z_{n_k}(T_{n_k})-\zeta\Vert_H = 0.$$
%Then, by the lower semi-continuity of $U^\alpha$,
%$$U(\zeta) \leq  \liminf_{k \rightarrow \infty}U(z_{n_k}(T_{n_k})) \leq r,$$ which contradicts (\ref{ineq 1 upper lemma 1}).
\end{proof}

\begin{lemma}\label{2 upper bound lemma 2}
    For any $r >0$, $\delta >0$ and $\eta>0$, let $\beta(r,\delta) $ be as in Lemma \ref{2 upper bound lemma 1}. Then there exists $\bar{N} \in \mathbb{N} $ large enough such that 
	$$u \in H_{\eta,r,\delta }(\bar{N})\  \text{implies}\  I_{\bar{N}}(u) \geq r,$$
	where the set $H_{\eta,r,\delta }(\bar{N})$ is defined for $\bar{N} \in \mathbb{N}$ by
	$$H_{\eta,r,\delta }(\bar{N}):= \left\{u \in C([0,\bar{N}];E), \| u(0)\Vert_E \leq \eta, \| u(j)-(m_L-\psi)\Vert_E \geq \beta , j =1,...,\bar{N}\right\}.$$
\end{lemma}
\begin{proof}
    Assume there exist $r >0$, $\delta >0$, and $\eta>0$ such that for any $n \in \mathbb{N}$ there exists $u \in H_{\eta,r,\delta }(n)$ with \begin{equation}\label{ineqn: contradiction}
        I_n(u) \leq r.
        \end{equation}
     Then, by Proposition \ref{prop:energy upper bound} and discussions in Lemma \ref{lemma:z-energy},  $\mathbf{E}^*_L( u(n)) \leq C_{r, \eta}$, for all $n \in \mathbb{N}$. Here, $C_{r,\eta}$ is a positive constant depending on $r$ and $\eta$.
    Now, for any $k \in \mathbb{N}$ we define 
    $$r_k:=\inf\left\{I_k(u):u \in C([0,k];E),\mathbf{E}^*_L( u(0)) \leq C_{r, \eta}, \| u(k)-(m_L-\psi)\Vert_E \geq \beta \right\}.$$ Notice that if $u \in H_{\eta,r,\delta }(n) $, then for any $k \in \mathbb{N}$, $k \leq n $, $\mathbf{E}^*_L( u(k)) \leq C_{r, \eta}$ and $\| u(k)-(m_L-\psi)\Vert_E \geq \beta $. Therefore, if we can show that there exists $\bar{k} \in \mathbb{N}$ such that $r_{\bar{k}}>0$, then for any $u \in C([0,n\bar{k}];E)$
    $$I_{n\bar{k}}(u) \geq nr_{\bar{k}},$$ which contradicts (\ref{ineqn: contradiction}). The remainder of the proof proceeds similarly to the proof of Lemma \ref{lemma : energy increment}.
    
    To finish our proof, it remains to show that there exists $\bar{k} \in \mathbb{N}$ such that $r_{\bar{k}}>0$. In fact, if $r_{\bar{k}}=0 $ for all $\bar{k}$, then there exists a sequence $$ \{u_i \in C([0,\bar{k}];E), \mathbf{E}^*_L( u_i(0)) \leq C_{r, \eta}, \| u_i(\bar{k})-(m_L-\psi)\Vert_E \geq \beta \}, i \in \mathbb{N},$$ such that \begin{equation}
        \lim_{i \rightarrow \infty}I_{\bar{k}}(u_i)=0.   
        \end{equation}
    %This implies that there exists $f_i$ with $\lim_{i \rightarrow \infty}\|f_i\Vert_{L^2([0,\bar{k}];H)}=0$ such that $u_i$ is the solution of the skeleton equation
   %$$ \partial_tu_i(t)= \Delta u_i(t)+ F(u_i(t)) + G(t,u_i(t))f_i(t).$$
  %According to (\ref{ineqn: skeleton}), \begin{equation} \label{u_i bound}
     % \sup_{t \geq 0, i \in \mathbb{N}}\|u_i\Vert_E < \infty.  \end{equation} Since level sets of $\mathbf{E}^*_L(\cdot)$ are compact, there exists $x \in E$ such that $\lim_{i \rightarrow \infty}\|u_i-x\Vert_E=0$.  Then, by standard estimates, \begin{align*}
      % \frac{1}{2}\frac{d}{dt}\|u_i-z_0^x\Vert_H^2 &\leq \langle\Delta(u_i-z_0^x), u_i-z_0^x\rangle_H + \langle F(u_i)-F(z_0^x), u_i-z_0^x\rangle\\ &+ \langle G(t,u_i(t))f_i(t), u_i-z_0^x \rangle  \\& \leq -\|u_i-z_0^x \Vert_{H^1}^2 + C\| u_i-z_0^x\Vert_H^2 + C\|f_i(t)\Vert_H^2.  
      % \end{align*} 
      % By Gronwall's inequality 
      % $$
          % \|u_i(t)-z_0^x(t)\Vert_H^2 + 2 \int_0^t \|u_i(s)-z_0^x(s) \Vert_{H^1}^2  \;ds \leq C_t(\|u_i(0)-x\Vert_H^2 + \|f_i\Vert_{L^2([0,t];H)}^2).  $$ 
  % Then, there exists $t \in [\frac12\bar{k}, \bar{k}]$ such that
   %\begin{equation}\label{ineqn:H^1 dist}
    %   \|u_i(t)-z_0^x(t) \Vert_{H^1}^2 \leq \frac{C_{\bar{k}}}{\bar{k}}(\|u_i(0)-x\Vert_H^2 + \|f_i(t)\Vert_H^2).   \end{equation}
     %  Thanks to Lemma \ref{lemma-energy decreasing}, for any $\eta >0$, one can choose $\bar{k}$ sufficiently large so that $\mathbf{E}^*_L(z_0^x(t)) \leq \frac{\eta}{4} $. Thus, by Sobolev inequality and (\ref{ineqn:H^1 dist}), 
     By the same arguments as in the proof of Lemma \ref{lemma : energy increment}, there exist $x \in E$ and a subsequence still denoted by $u_i(0)$ such that $$\lim_{i \rightarrow \infty}\|u_i(0)-x\Vert_H  =0.$$     
     for any $\eta>0$, there exist $\bar{k}$ sufficiently large, $N \in \mathbb{N}$ and some $t \in[\frac{\bar{k}}{2}, \bar{k}] $ such that for any $ i \in \mathbb{N} $, $i \geq N$
   $$\mathbf{E}^*_L(u_i(t)) \leq \frac{\eta}{2}. $$ 
  Furthermore, by (\ref{Energy inequality}) and (\ref{u_i bound}), \begin{align*}
\mathbf{E}^*_L(u_i(\bar{k}))&\leq \mathbf{E}^*_L(u_i(t)) + C(1+\sup_{s \geq 0, i \in \mathbb{N}}\|u_i(s)\Vert_E^2)\cdot I_{\bar{k}}(u_i) \\ &\leq \mathbf{E}^*_L(u_i(t))+C I_{\bar{k}}(u_i).  \end{align*} Since $\lim_{i \rightarrow \infty}I_{\bar{k}}(u_i)=0 $, for any $\eta >0$ there exists $N$ such that for all $i \geq N$, $CI_{\bar{k}}(u_i) \leq \frac{\eta}{2}$.
 Therefore, one obtains that, for $i$ sufficiently large, 
       $$ \mathbf{E}^*_L(u_i(\bar{k})) \leq \eta. $$ Moreover, by Proposition \ref{prop: small energy}, for any $\beta >0$, one can choose $\eta$ sufficiently small such that $$\|u_i(\bar{k})-(m_L-\psi)\Vert_E \leq \frac{\beta}{2} $$which contradicts the assumption that $ \|u_i(\bar{k})-(m_L-\psi)\Vert_E  \geq \beta $. The proof is complete.
\end{proof}

\begin{theorem}\label{prop:upper bound}
	 Under the Assumption \ref{assumption 3} and the assumption of Proposition \ref{prop:exponential estimate}, the family of invariant measures $\left\{\bar{\mu} _\epsilon \right\}_{\epsilon >0}$ satisfies a large deviation principle upper bound in $E$ with the rate function $U_L$. That is, for any $r \geq 0$ and $\delta , \gamma  >0$, there exists $\epsilon _0>0$ such that 
$$\bar{\mu} _\epsilon \left(\left\{h \in E: dist_E(h,\Phi_L (r)) \geq\delta \right\}\right) \leq \exp\left(-\frac{r-\gamma }{\epsilon }\right), \ \epsilon < \epsilon _0.$$
\end{theorem}
\begin{proof}
	For any fixed $r>0$, $\delta >0$ and $\gamma >0$, let $R_r$, $\epsilon _r$ $k^\star$ and $p^\star$ be as in Proposition \ref{prop:exponential estimate}. Since $\bar{\mu} _\epsilon $ is the invariant measure for equation (\ref{ACE2}), for $t \geq  \bar{N}$ we have 
	\begin{align*}
		&\ \ \ \ \ \bar{\mu} _\epsilon \left(\left\{h \in E: dist_E(h,\Phi_L (r)) \geq \delta \right\}\right) \\&= \int_{E}\mathbb{P} \left(dist_E(\bar{u}_\epsilon ^y(t),\Phi_L (r)) \geq \delta \right)  \,d\bar{\mu} _\epsilon (y) \\ &= \int_{B_{W^{k^\star,p^\star}}^c(0,R_r)}\mathbb{P} \left(dist_E(\bar{u}_\epsilon ^y(t),\Phi_L (r)) \geq \delta \right)  \,d\bar{\mu} _\epsilon (y) \\ &\ \ + \int_{\overline{B_{W^{k^\star,p^\star}}(0,R_r) }}\mathbb{P} \left(dist_E(\bar{u}_\epsilon ^y(t),\Phi_L (r)) \geq \delta, u_\epsilon ^y \in H_{CR_r,r,\delta }(\bar{N}) \right)  \,d\bar{\mu} _\epsilon (y) \\  &\ \ + \int_{\overline{B_{W^{k^\star,p^\star}}(0,R_r) }}\mathbb{P} \left(dist_E(\bar{u}_\epsilon ^y(t),\Phi_L (r)) \geq \delta, \bar{u}_\epsilon ^y \notin H_{CR_r,r,\delta }(\bar{N}) \right)  \,d\bar{\mu} _\epsilon (y) \\ &=: K_1 +K_2 + K_3,
	\end{align*}
	where $\bar{N}$ is chosen as in Lemma \ref{2 upper bound lemma 2} such that $$u \in H_{CR_r,r,\delta }(\bar{N}) \ \text{implies} \ I_{\bar{N}}(u) \geq r.$$
	Now, thanks to Theorem \ref{thm:exponential-tightness}, we have for $ 0<\epsilon \leq \epsilon _r $,
	$$K_1 \leq \bar{\mu} _\epsilon \left(B^c_{ W^{k^\star,p^\star} }(0,R_r)\right) \leq \exp(-\frac{r}{\epsilon }).$$ To estimate $K_2$, notice that $\overline{B_{W^{k^\star,p^\star} }(0,R_r) }$ is a compact subset of E, and $H_{CR_r,r,\delta }(N)$ is a closed set in $C([0,N];E)$. Then, by Corollary \ref{corollary DZ2}, $\left\{\mathcal{L} (\bar{u}_\epsilon^x)\right\}_{\epsilon > 0}$ satisfies a Dembo-Zeitouni uniform large deviation principle over compact sets. We infer that there exists $\epsilon _1 >0$ such that \begin{align*}
		K_2 &\leq \sup_{y \in \overline{B_{W^{k^\star,p^\star} }(0,R_r) }  }\mathbb{P} \left( \bar{u}_\epsilon ^y \in H_{CR_r,r,\delta }(\bar{N}) \right) \\ &\leq \exp\left(-\frac{1}{\epsilon }\left[\inf_{z \in \overline{B_{W^{k^\star,p^\star} }(0,R_r) }  }\inf_{h \in H_{CR_r,r,\delta }(\bar{N})}I_N^z(h) - \gamma \right]\right)\\ &\leq \exp\left(-\frac{r-\gamma }{\epsilon }\right),
	\end{align*}
	for any $\epsilon \leq \epsilon _1$. Here, the third inequality is due to Lemma \ref{2 upper bound lemma 2}. 

	Finally, let us deal with the third term $K_3$. By the definition of $H_{CR_r,r,\delta }(N)$ and the Markov property of $u_\epsilon $, we have 
	\begin{align*}
		K_3 &\leq \int_{ B_{W^{k^\star,p^\star} }(0,\bar{R}_r) }\mathbb{P} \left(\bigcup _{j=1}^{\bar{N}}\left\{\| \bar{u}_\epsilon ^y(j)-(m_L-\psi)\Vert _E < \beta \right\}\bigcap \left\{dist_E(\bar{u}_\epsilon^y(t), \Phi_L (r)) \geq \delta \right\}\right) \,d\bar{\mu} _\epsilon (y) \\ &\leq \sum_{j =1}^N \int_{ B_{W^{k^\star,p^\star} }(0,\bar{R}_r) }\mathbb{P} \left(\left\{\| \bar{u}_\epsilon ^y(j)-(m_L-\psi)\Vert _E < \beta \right\}\bigcap \left\{dist_E(\bar{u}_\epsilon^y(t), \Phi_L (r)) \geq \delta \right\}\right) \,d\bar{\mu} _\epsilon (y) \\ &\leq \sum_{j=1}^N \sup_{y \in B_E(m_L-\psi,\beta )}\mathbb{P} \left(dist_E(\bar{u}_\epsilon ^y(t-j),\Phi_L (r)) \geq \delta \right).
	\end{align*}

	In order to use the ULDP, we need to transfer the event at $t-j$ to an event in $C([0,t-j];E)$. To achieve that, we choose $t \geq T + N$, where $T$ was given in Lemma \ref{2 upper bound lemma 1}. We then have that for $y \in B_E(m_L-\psi,\beta )$, if
	$$dist_{C([0,t-j];E)}\left(\bar{u}_\epsilon^y, K ^y(r)\right) < \frac{\delta }{2},$$ then $$ \inf\left\{\| \bar{u}_\epsilon^y-v\Vert _{C([0,t-j];E)}: v \in C([0,t-j];E),\  \| v(0)-(m_L-\psi)\Vert _E <\beta , \ I_{t-j}(v) \leq r \right\} < \frac{\delta }{2},$$ where $$K ^y(r) := \left\{v \in C([0,t-j];E): v(0) =y, I_{t-j}(v) \leq r\right\}.$$ It turns out by applying Lemma \ref{2 upper bound lemma 1} for $\frac{\delta}{2}$ that $$ dist_E(\bar{u}_\epsilon^y(t-j), \Phi_L (r)) <\delta.$$ 
Thus, according to the FWULDP results, there exists $\epsilon _{0,j}>0$ such that for any $\epsilon \leq \epsilon _{0,j}$,
\begin{align*}
	\sup_{y \in B_E(m_L-\psi,\beta )}\mathbb{P} \left(dist_E(\bar{u}_\epsilon ^y(t-j),\Phi_L (r)) \geq \delta \right) & \leq \sup_{y \in B_E(m_L-\psi,\beta )}\mathbb{P} \left(dist_{C([0,t-j];E)}\left(\bar{u}_\epsilon^y, K ^y(r)\right) \geq \frac{\delta }{2}\right) \\ &\leq \exp\left(-\frac{r-\gamma }{\epsilon }\right).
\end{align*}
Therefore, we choose $\epsilon_0= \min(\epsilon _r,\epsilon _1,\epsilon _{0,1},...,\epsilon _{0,N})>0$, it follows that for any $\epsilon \leq \epsilon _0$
$$K_3 \leq \bar{N} \exp\left(-\frac{r-\gamma }{\epsilon }\right),$$
Combining the estimates of $K_1$, $K_2$, $K_3$, we complete the proof of this proposition.
\end{proof}

\begin{theorem}
Under Assumption \ref{assumption 3} and the growth assumption of
Proposition \ref{prop:exponential estimate}, the family of invariant measures
$\{\bar\mu_\epsilon\}_{\epsilon>0}$ satisfies an LDP in $E$
with a good rate function $U_L$.
\end{theorem}    

With the upper bound of LDP for $\bar{\mu}_\epsilon$, it is easy to see the following result.
\begin{corollary}
   Under Assumption \ref{assumption 3} and the assumption of Proposition \ref{prop:exponential estimate}, there exist constants $C>0$ and $\delta_0>0$ such that for every $0<\delta<\delta_0$ there exists $\epsilon_0 >0$ such that
    $$\mu_\epsilon(\{h \in \mathcal{E}:\|h-m_L\Vert_E \geq \delta \}) \leq \exp\left(-\frac{C\delta^2}{\epsilon} \right), \epsilon <\epsilon_0.$$ In particular, the measures $\mu_\epsilon$ concentrate around the  minimiser $m_L$ exponentially fast when $\epsilon$ is sufficiently small.
    \end{corollary}
    \begin{proof}
        This is a direct consequence of Propositions \ref{prop: small energy} and \ref{prop:energy upper bound}, and the LDP upper bound for $\bar{\mu}_\epsilon$ given in Proposition \ref{prop:upper bound}.
    \end{proof}

  \appendix
\renewcommand\thesection{\normalsize Acknowledgements}
\section{}
We acknowledge the financial supports from an EPSRC grant (ref EP/S005293/2), and a Royal Society Newton Fund grant (ref. NIF\textbackslash R1\textbackslash 221003).

\addtolength{\itemsep}{-1.5 em} % ????????????
\setlength{\itemsep}{-3pt}
\footnotesize

\phantomsection
	\addcontentsline{toc}{section}{References}
    
\bibliographystyle{siam}
%\newpage
\footnotesize
\bibliography{Ref}
\end{document}